\documentclass[11pt,a4paper,oneside,onecolumn,number,preprint]{article}

\usepackage[utf8]{inputenc}
\usepackage[T1]{fontenc}
\usepackage{authblk}
\usepackage{amsmath}
\usepackage{amsfonts,amssymb}
\usepackage{algorithm}
\usepackage{algpseudocode}
\usepackage{bbm}
\usepackage{xcolor}
\usepackage{soul}
\usepackage{graphicx}
\usepackage{lmodern}
\usepackage{wrapfig}
\usepackage{subcaption}
\usepackage[nopar]{lipsum}
\usepackage{amsthm}
\usepackage{euscript}
\usepackage{blkarray}
\usepackage{makecell, multirow, tabularx,booktabs}
\usepackage{enumitem}
\usepackage{mathtools}
\usepackage{soul}
\usepackage{amscd}
\usepackage{float}
\usepackage{geometry}
\usepackage[pagebackref,colorlinks=true,
linkcolor=blue,
urlcolor=blue,colorlinks,
citecolor=green]{hyperref}
\usepackage{cleveref}
\usepackage{geometry}
\usepackage{pgf,tikz,pgfplots}
\tikzstyle{none}=[]
\tikzstyle{new style 0}=[draw,circle,fill=white]
\tikzstyle{new edge style 1}=[draw,dashed]
\tikzstyle{new edge style 1}=[draw,dashed]
\usepackage{adjustbox}
\usepackage{setspace}
\usepackage[numbers,sort&compress]{natbib}

\allowdisplaybreaks[4]
\usetikzlibrary{calc}
\usetikzlibrary{patterns}
\pgfplotsset{compat=1.14}
\numberwithin{equation}{section}
\newtheorem{thm}{Theorem}[section]
\newtheorem{prop}[thm]{Proposition}

\newtheorem{cor}[thm]{Corollary}
\newtheorem{exm}[thm]{Example}
\newtheorem{df}[thm]{Definition}
\newtheorem{rem}[thm]{Remark}

\providecommand{\keywords}[1]
{
  \small	
  \textbf{\textit{Keywords---}} #1
}
\providecommand{\subject}[1]
{
  \small	
  \textbf{\textit{MSC2020---}} #1
}
\title{A Study of Hypergraph Using Null Spaces of the Incidence Matrix and its Transpose.}
\author{Samiron Parui\thanks{School of Mathematical Sciences,  National Institute of Science Education and Research Bhubaneswar,  Bhubaneswar, Padanpur, Odisha 752050, India\\
Email: \href{mailto:me@samironparui@gmail.com}{samironparui@gmail.com}
}}

\date{\today}
\begin{document}

\maketitle
\hrule
\begin{abstract}In this study, we explore the substructures of a hypergraph that lead us to linearly dependent rows (or columns) in the incidence matrix of the hypergraph.
     These substructures are closely related to the spectra of various hypergraph matrices, including the signless Laplacian, adjacency, Laplacian, and adjacency matrices of the hypergraph's incidence graph. Specific eigenvectors of these hypergraph matrices serve to characterize these substructures. We show that vectors belonging to the nullspace of the adjacency matrix of the hypergraph's incidence graph provide a distinctive description of these substructures. Additionally, we illustrate that these substructures exhibit inherent similarities and redundancies, which manifest in analogous behaviours during 
    random walks and similar values of hypergraph centralities.
\end{abstract}


\subject{
		05C65, 
  05C50 
		; Secondary
    05C81, 
    05C69 
	}
\vspace{5pt}\hrule

\section{Introduction}
\label{sec:introduction}
A hypergraph $H$ is an ordered pair $(V(H),E(H))$. The non-empty set, $V(H)$, is called the vertex set, and the hyperedge set, $E(H)$, is a subset of $\mathcal{P}({V(H)})$, the power set of $V(H)$. For any vertex $v\in V(H)$, the collection $E_v(H)=\{e\in E(H):v\in e\}$ is referred to as the \emph{star of the vertex} $v$. 
The references encompassing fundamental terminology and concepts related to hypergraphs include  \cite{Berge-hypergraph,Berge-graph-hypergraph,Voloshin-Vitaly-graph-hypergraph,Bretto-hypergraph}. A hypergraph $H$ is a \emph{finite hypergraph} if the vertex set $V(H)$ contains a finite number of elements. In this work, we consider finite hypergraphs only. Hypergraphs are investigated using various matrices, such as adjacency matrix, Laplacian matrix, signless Laplacian matrix etc. (\cite{hg-mat,Sarkar-banerjee-2020,my1st,rodriguez2003laplacian,Trevisan_signless,Swarup-panda-2022-hypergraph}). Here, we study the relation between a hypergraph $H$, and the null space of the incidence matrix $I_H$. Certain non-zero vectors in the null space of $I_H$ and in the null space its transpose $I_H^T$ tell us about particular structures within the hypergraph $H$. The adjacency matrix $A_{G_H}$ of the incidence graph $G_H$ of the hypergraph $H$ is such that $$ A_{G_H}=\left[\begin{smallmatrix}
            0_{n}& I_H\\
            I_H^T&0_{m}
        \end{smallmatrix}\right].$$
        Therefore, any vector in the null space of both the matrices $I_H$ and $I_H^T$ can be embedded in the null space of $A_{G_H}$. Given any non-zero vector $y$ in the null space of $A_{G_H}$, we have $I_Hy|_{E(H)}=0$, and $I_H^Ty|_{V(H)}=0$. Here $y|_{E(H)}$, and $y|_{V(H)} $ are the restrictions of $y$ in $V(H)$, and $E(H)$, respectively. Since $y$ is non-zero, at least one of $y|_{E(H)}$, and $y|_{V(H)}$ is non-zero. Therefore, any non-zero element of the null space of $A_{G_H}$ corresponds to at least one non-zero element in the null space of $I_H$ or $I_H^T$. A non-zero element in the null space of $I_H^T$ and $I_H$ corresponds to a set of linearly dependent rows and columns, respectively, in the incidence matrix $I_H$. Each row and column of $I_H$ corresponds to a vertex and hyperedge, respectively, in $H$. If a subset of vertices $U\subseteq V(H)$ is such that the corresponding collection of rows in $I_H$ is linearly dependent, then we refer to $U$ as a \emph{linearly dependent set of vertices}. Similarly, a \emph{collection of hyperedges $E\subseteq E(H)$ is linearly dependent} if the corresponding collection of columns is linearly dependent.
        
        The presence of linearly dependent sets of vertices and hyperedges unveils intriguing substructures within hypergraphs. A fundamental instance arises when the vertex-edge incidence matrix $I_H$ has identical rows. That is, a collection of vertices exists such that each vertex of the collection has the same star. Given $E\subseteq E(H)$, if $W_E\subseteq V(H)$ is such that $E_v(H)=E $ for all $v\in W_E$ and $E_u(H)\ne E$ for all $u\notin W_E$, then the collection $W_E$ is called \emph{unit}. The relation between the unit and the spectra of different hypergraph matrices like signless Laplacian, Laplacian, and adjacency are studied in \cite{unit}. In this article, we show that the null space of $I_H^T$ reflects the existence of units in the hypergraph $H$. If $W\subseteq V(H)$ is a maximal set with the property that given any function\footnote{For any finite hypergraph $H$, with $n$ vertices and $m$ hyperedges, one convention is to enumerate the vertex set $V(H)=\{v_1,v_1,\ldots,v_n\}$, and the hyperedge set $E(H)=\{e_1,e_2,\ldots,e_m\}$, and represent the incidence matrix $I_H$ as an $n\times m$ matrix whose $(i,j)$-th element corresponds to the vertex $v_i$, and the hyperedge $e_j$. We avoid the enumeration, and the rows and columns of the $I_H$ are indexed by $V(H)$ and $E(H)$, respectively. As a consequence, instead of being column vectors in $\mathbb{R}^m$ and $\mathbb{R}^n$ the elements of the null space of $I_H$ and $I_H^T$ are, respectively, the function of the form $x:E(H)\to\mathbb{R}$ and $y:V(H)\to\mathbb{R}$.} $x:V(H)\to\mathbb{R}$  with $x(u)=0$ for all $u\notin W$, and $\sum_{u\in W}x(u)=0$ we have $I_H^Tx=0$, then $W$ is a unit in $H$ (see the \Cref{null-ag}). 
        Not only units but any linearly dependent collection of vertices is manifested in the null space of $I^T_H$, and hence in the null space of $A_{G_H}$(see the \Cref{Ix=0_ldv}, and the \Cref{0-ev-ld}). Similarly, the \Cref{I_H_lde} shows that any linearly dependent collection of hyperedges is reflected in the null space of $I_H$. In the \Cref{par-part}, we show that given two disjoint subsets $U$, and $V$ of the vertex set $V(H)$, if the difference of characteristic function\footnote{Given any $U\subseteq V(H)$, the \emph{characteristic function} $\chi_U:V(H)\to\mathbb{R}$ is such that $\chi_U(u
        )=1$ if $u\in U$, and otherwise $\chi_U(u
        )=0$.} $\chi_U-\chi_V$ belongs to the null space of $I_H^T$, then $|U\cap e|=|V\cap e|$ for all $e\in E(H)$. We say $U$ and $V$ form an equal partition of hyperedges in $H$. 
The \Cref{app-matrices} shows units, equal partitions of hyperedges, and other collections of linearly dependent vertices are related to the spectra of the signless Laplacian, adjacency, and Laplacian matrices associated with hypergraph $H$. Specific linearly dependent collections of vertices such as unit, an equal partition of hyperedges exhibit similarity in random walks on hypergraphs (see the \Cref{rwh}) and in some hypergraph centrality (see \Cref{cent}). The \Cref{count} shows that some counting methods, like the Pigeonhole Principle and the principle of inclusion-exclusion, can be proved using the linearly dependent collection of vertices and hyperedges.

\section{Linearly independent  set of vertices and hyperedges}\label{lindepve}
The incidence matrix $I_H=\left[i_{ve}\right]_{v\in V(H),e\in E(H)}$ of a hypergraph $H$ is the matrix with its rows and columns indexed by the vertices and hyperedges of the hypergraph $H$, respectively. Each entry $i_{ve}$ is 1 if the vertex $v\in e$ otherwise $i_{ve}=0$.  
For a hyperedge $e\in E(H)$, the $e$-th column of $I_H$ corresponds to the characteristics function of $e$, that is $\chi_{_e}:V(H)\to\{0,1\}$ defined by $\chi_{_e}(v)=i_{ve}$ for all $v\in V(H)$. Similarly, for $v\in V(H)$, the $v$-th row of $I_H$ gives the characteristic function of $E_v(H)$, which is $s_v:E(H)\to\{0,1\}$ such that $s_v(e)=i_{ve}$. We call a collection of vertices in $H$ a \emph{linearly dependent set of vertices} if the corresponding collection of rows in $I_H$ is linearly dependent. Similarly, a collection of hyperedges in $H$ is \emph{a linearly dependent set of hyperedges} in $H$ if the corresponding collection of columns in $I_H$ is linearly dependent.
 \begin{df}[Linearly dependent set of vertices and hyperedges]
  Let $H$ be a hypergraph.
 Given any $U\subseteq V(H)$, we say $U$ is \emph{linearly independent (dependent)}  if $\{s_v:v\in U\}$ is linearly independent (dependent).
      A collection of hyperedges $E\subseteq E(H)$ is \emph{linearly independent (dependent)} if $\{\chi_e:e\in E(H)\}$ is linearly independent (dependent).
 \end{df}
 Given a linearly dependent set of vertices $U=\{v_1,v_2,\ldots,v_k\}\subseteq V(H)$, there exists $c_1,c_2,\ldots c_k\in \mathbb{R}$ (not all are zeroes simultaneously) such that $\sum_{i=1}^kc_is_{v_i}=0$. This gives us a \emph{coefficient vector} $c_U:V(H)\to\mathbb{R}$ defined by
 \[c_U(v)=\begin{cases}
     c_i&\text{~if~}v=v_i\text{~for~}i=1,2,\ldots,k,\\
     0&\text{~if~}v\notin U.
 \end{cases}\]
 For any linearly dependent set of vertices \(U\) within graph \(H\), the coefficient vector \(c_U\) is not unique. Nevertheless, owing to the relation \(\sum_{i=1}^k c_i s_{v_i} = 0\), it follows that, for any coefficient vector \(c_U\) associated with \(U\), the equality
\( \sum_{v \in e} c_U(v) = \sum_{i=1}^k c_U(v_i) s_{v_i}(e) = 0 \)
holds for all \(e \in E(H)\). For any function $x:V\to\mathbb{R}$ on a non-empty set $V$, the \emph{support of $x$} is $supp(x)=\{v\in V: x(v)\ne 0\}$. For any linearly dependent set of vertices $U$ in $H$ with a coefficient vector $c_U$, the support of the coefficient vector, $supp(c_U)\subseteq U$.
 This leads to the following result.
 \begin{prop}\label{Ix=0_ldv}
     Let $H$ be a hypergraph. A subset $U\subseteq V(H)$ is linearly dependent if and only if $I_H^Tx=0$ for a non-zero vector $x:V(H)\to\mathbb{R}$ such that $supp(x)\subseteq U$.
 \end{prop}
 \begin{proof}
     If $U=\{v_1,\ldots,v_k\}\subseteq V(H)$ is linearly dependent, then there exists a non-zero co-efficient vector $c_U:V(H)\to\mathbb{R}$ such that $(\sum\limits_{v\in U}c_U(v)s_v)=0$. Since $supp(c_U)\subseteq U$,
     $$(I_H^Tc_U)(e)=\sum\limits_{v\in V(H)}i_{ve}c_U(v)=(\sum\limits_{v\in U}c_U(v)s_v)(e)=0 \text{~for all~}e\in E(H).$$
     Therefore, $I_H^Tc_U=0$. Conversely, if $I_H^Tx=0$ with $supp(x)\subseteq U$, then for all $e\in E(H)$,
     $$(\sum\limits_{v\in U}x(v)s_v)(e)= (I_H^Tx)(e)=0.$$ Therefore, $U$ is a linearly dependent set of vertices in $H$.
 \end{proof}
 
 Similarly, for any linearly dependent set of hyperedges $E=\{e_1,e_2,\ldots,e_k\}\subseteq E(H)$, there exists $c_1,c_2,\ldots,c_k$ (not all are zeroes simultaneously) such that $\sum_{i=1}^kc_i\chi_{e_i}=0$ and this leads us to a coefficient vector $c_E:E(H)\to\mathbb{R}$( is not unique ) such that
 \[c_E(e)=\begin{cases}
     c_i&\text{~if~}e=e_i\text{~for~}i=1,2,\ldots,k,\\
     0&\text{~if~}e\notin E.
 \end{cases}\]
 The coefficient vector $c_E$ leads us to a result similar to $\Cref{Ix=0_ldv}$ for linearly dependent hyperedges.
\begin{prop}\label{I_H_lde}
    Let $H$ be a hypergraph. A collection of hyperedge $E\subseteq E(H)$ is linearly dependent if and only if $I_Hy=0$ for a non-zero vector $y:E(H)\to\mathbb{R}$ with $supp(y)\subseteq E$.\end{prop}
\begin{proof}
    If $E$ is linearly dependent, then there exists a coefficient vector $c_E:E(H)\to\mathbb{R}$ such that $\sum\limits_{e\in E(H)}c_E(e)\chi_{e}=0$. For all $v\in V(H)$, we have  $(I_Hc_E)(v)=\sum\limits_{e\in E(H)}i_{ve}c_E(e)=\sum\limits_{e\in E(H)}c_E(e)\chi_{e}(v)=0$. Therefore, $I_Hc_E=0$. Conversely, If $I_Hy=0$ with $supp(y)\subseteq E$, then $(\sum\limits_{e\in E}y(e)\chi_{e})(v)=(I_Hy)(v)=0 $ for all $v\in V(H)$. Therefore, $E$ is a linearly dependent collection of hyperedges in $H$.
\end{proof}
Given a rectangular matrix $M$, the row rank and column rank of $M$ are the same. This leads to the following result.
 \begin{thm}\label{lin-ind-v=e}
    A hypergraph contains $r$ linearly independent vertices if and only if it contains $r$ linearly independent hyperedges.
 \end{thm}
 \begin{proof}
     Let $H$ be a hypergraph. If $v_1,\ldots,v_r$ are the $r$ linearly independent vertices in $H$, then $s_{v_1},\ldots,s_{v_r}$ are linearly independent vectors. Consider the vectors $(I_Hs_{v_i}):V(H)\to \mathbb{R}$ for all $i=1,\ldots, r$. We claim that $\{I_Hs_{v_i}:i=1,\ldots,r\}$ are linearly independent. Assuming otherwise, if $I_Hs_{v_1},\ldots, I_Hs_{v_r}$ are linearly dependent vectors then
     $$c_1I_Hs_{v_1}+\ldots +c_rI_Hs_{v_r}=0$$ for some $c_1,\ldots,c_r\in \mathbb{R}$ (not all $0$ simultaneously).  The vector $y=c_1s_{v_1}+\ldots +c_rs_{v_r}\in\langle \{s_{v_i}:i=1,\ldots,r\}\rangle$, the vector space spanned by $\{s_{v_i}:i=1,\ldots,r\}$. Since $I_Hy=0$, thus, $\sum\limits_{e\in E(H)}s_{v_j}(e)y(e)=\sum\limits_{e\in E(H)}i_{v_je}y(e)=(I_Hy)(v_j)=0$
     for all $j=1,\ldots,r$. 
     Consequently, $y $ is orthogonal to $s_{v_i}$ for all $i=1,\ldots,r$. That is, $y\in \langle \{s_{v_i}:i=1,\ldots,r\}\rangle^\perp$, the orthogonal complement of the vector space $\langle \{s_{v_i}:i=1,\ldots,r\}\rangle$. Therefore, $c_1s_{v_1}+\ldots +c_rs_{v_r}=y=0$, which is a contradiction to the fact that $s_{v_1},\ldots,s_{v_r}$ are linearly independent vectors.
     Therefore, $\{I_Hs_{v_i}:i=1,\ldots,r\}$ are linearly independent vectors. Since $ I_Hs_{v_i}=\sum\limits_{e\in E(H)}s_{v_i}(e)\chi_{e}$
     for any $i=1,\ldots,r$, the $r$ linearly independent vectors $\{I_Hs_{v_i}:i=1,\ldots,r\}$
     belongs to a vector space spanned by the vectors $\{\chi_{e}:e\in E(H)\}$. Consequently, at least $r$ linearly independent hyperedges exist in $H$.

 Conversely, if $H$ contains $r$ linearly independent hyperedges $e_1,\ldots,e_r$, then similarly we can show that $I_H^T\chi_{e_1},\ldots,I_H^T\chi_{e_r}$ are linearly independent. Since $I_H^T\chi_{e_i}=\sum\limits_{u\in V(H)} \chi_{e_i}(u)s_u$ for all $i=1,\ldots,r$, the linearly independent vectors $\chi_{e_1},\ldots,\chi_{e_r}$ belongs to the vector space spanned by $\{s_u:u\in V(H)\}$. Consequently, at least $r$ linearly independent vertices exist in $H$.
 \end{proof}

\subsection{Adjacency matrix of incidence graph}\label{A_G_H}
Given a hypergraph $H$, the \emph{Incidence graph} $G_H$ of the hypergraph $H$ is a bipartite graph such that $V(G_H)=V(H)\cup E(H)$, and $E(G_H)=\{\{v,e\}\in V(H)\times E(H):v\in e\}$.
The adjacency matrix $A_G$ of a graph\footnote{A hypergraph $H$ is called $m$-uniform if $|e|=m$ for all $e\in E(H)$. Here, graph means $2$-uniform hypergraph. That is, we are considering only simple graphs without any loop.} $G$ on $n$ vertices is an $n\times n$ matrix whose rows and columns are indexed by the vertex set $V(G)$ of $G$, and $A_G=\left(a_{uv}\right)_{u,v\in V(G)}$ such that for all distinct $u,v\in V(G)$, $a_{uv}=1$ if $\{u,v\}\in E(G)$, and otherwise $a_{uv}=0$. The diagonals entry $a_{uu}=0$ for all $u\in V(H)$. Given a hypergraph $H$ its dual $H^*$ is a hypergraph such that $V(H^*)=E(H)$, and $E(H^*)=\{E_v(H):v\in V(H)\}$.
The Incidence graph $G_H$ uniquely determines the collection $\{H, H^*\}$ containing the hypergraph $H$ and its dual $H^*$. Thus, the spectrum of $A_{G_H}$ is expected to encrypt certain information about $H$.
     \subsubsection{\texorpdfstring{Zero eigenvalues of $A_{G_H}$}{Zero Eigenvalues of the Adjacency Matrix of Incidence Graph }}
       The adjacency matrix $A_{G_H}$ of the incidence graph $G_H$ can be represented as
    \begin{align}\label{adj-incident}
        A_{G_H}=\left[\begin{smallmatrix}
            0_{n}& I_H\\
            I_H^T&0_{m}
        \end{smallmatrix}\right],   
    \end{align}
     where $0_{p}$ is a $p\times p$ matrix with all its entry $0$ for all $p\in\mathbb{N}$, and $I_H$ is the incidence matrix of the hypergraph $H$.  
     Given two functions $y:E(H)\to\mathbb{R}$, and $z:V(H)\to \mathbb{R}$, the embedding of $y$, and $z$ in $\mathbb{R}^{V(G_H)}$ are respectively
the functions $\Tilde{y}:V(G_H)\to\mathbb{R}$, and $\Tilde{z}:V(G_H)\to\mathbb{R}$ defined by 
\[\Tilde{y}(u)=
\begin{cases}
   0&\text{~if~}u\in V(H),\\
   y_i(u)&\text{~if~}u\in E(H),
\end{cases}
\text{~and~}
\Tilde{z}(u)=
\begin{cases}
   0&\text{~if~}u\in E(H),\\
   z_i(u)&\text{~if~}u\in V(H).
\end{cases}\]
     Given any function $x:V(G_H)\to\mathbb{R}$, if  $x|_{_{V(H)}}:V(H)\to\mathbb{R}$, and  $x|_{_{E(H)}}:E(H)\to\mathbb{R}$ are the restriction of  $x$ to $V(H)$, and $E(H)$ respectively, then $(A_{G_H}x)(v)=({I_Hx|_{_{E(H)}}})(v)$ for all $v\in V(H)$, and $(A_{G_H}x)(e)=(I_H^Tx|_{_{V(H)}})(e)$ for all $e\in E(H)$.
     Thus $I_H^T z=0$ if and only if $A_{G_H}\Tilde{z}=0$ for any $z:V(H)\to \mathbb{R}$, and $I_Hy=0$ if and only if $ A_{G_H}\Tilde{y}=0$ for any $y:E(H)\to\mathbb{R}$.
     This leads us to the following result.
     \begin{thm}\label{0-ev-ld}
         Let $H$ be a hypergraph. The matrix $A_{G_H}$
has a $0$-eigenvalue if and only if $H$ has at least one of the following.
\begin{enumerate}
    \item A linearly dependent set of vertices.
    
    \item A linearly dependent set of hyperedges.
  \end{enumerate}
  \end{thm}
 \begin{proof}
        If  $A_{G_H}$ has a $0$ eigenvalue then there exists a non-zero vector $x:V(G_H)\to\mathbb{R}$ such that $A_{G_H}x=0$. Therefore, $I_Hx|_{_{E(H)}}=0$, and $I_H^Tx|_{_{V(H)}}=0$. Since $x$ is non-zero, at least one of  $x|_{_{E(H)}}$ and $x|_{_{V(H)}}$ is non-zero. Consequently, by \Cref{Ix=0_ldv}, and \Cref{I_H_lde}, there exists a linearly dependent set of vertices or a linearly dependent set of hyperedges.
        
        Conversely, if $U\subseteq V(H)$ is a linearly dependent set of vertices with a coefficient vector $c_U:V(H)\to\mathbb{R}$, then by \Cref{Ix=0_ldv}, $I_H^Tc_U=0$. Consequently, $A_{G_H}\tilde c_U=0$. Similarly, by \Cref{I_H_lde}, given any linearly dependent set of hyperedges with a coefficient function $c_E:E(H)\to\mathbb{R}$, we have $A_{G_H}\tilde c_E=0$.
    \end{proof}
    \begin{rem}\label{rem Ax=0}\rm
        As indicated in \Cref{0-ev-ld}, a $0$ eigenvalue of $A_{G_H}$ implies one of the following three specific situations. Corresponding eigenvectors can characterize the situations.
        \begin{itemize}\rm
          \item \textbf{Linearly dependent set of vertices}: If $A_{G_H}x=0$ with $supp(x)\subseteq V(H)$, then by the \Cref{Ix=0_ldv}, there exists a linearly dependent set of vertices $U$ in $H$, such that $supp(x)\subseteq U$.
            \item \textbf{Linearly dependent set of hyperedges:} If $A_{G_H}x=0$ with $supp(x)\subseteq E(H)$, then $I_Hx|_{_{E(H)}}=0$, and by \Cref{I_H_lde}, there exists a linearly dependent set of hyperedge $E$ with $supp(x)\subseteq E$.
            \item If $A_{G_H}x=0$ with $supp(x)\cap E(H)\ne\emptyset $, and $supp(x)\cap V(H)\ne\emptyset$, then $I_Hx|_{_{E(H)}}=0$, and $I_H^Tx|_{_{V(H)}}=0$. Thus, by \Cref{Ix=0_ldv}, and \Cref{I_H_lde}, this eigenvector indicate both a linearly dependent set of vertices $U$ and a linearly dependent set of hyperedges $E$ in the hypergraph $H$ with $supp(x|_{V(H)})\subseteq U$, and  $supp(x|_{E(H)})\subseteq U$.
        \end{itemize}
    \end{rem}
  
  In any bipartite graph $G$, If two partite sets are not equipotent, then $A_G$ is a singular matrix. This leads us to the following result.
\begin{cor}
\label{v=e}
    Let $H$ be a hypergraph. If $A_{G_H}$ is non-singular, then $|E(H)|=|V(H)|$.
\end{cor}
\begin{proof}
   To prove this result, it is enough to show if $|E(H)|\ne|V(H)|$, then $A_G$ is singular, that is, $\det(A_{G_H})=0$. 

If $|V(H)|<|E(H)|$, then by \Cref{lin-ind-v=e}, the number of linearly independent hyperedges in $H$ is at most $|V(H)|$. Thus, $E(H)$ is a linearly dependent collection of hyperedges.
If $|E(H)|<|V(H)|$, then similarly, we can show that $V(H)$ is a collection of linearly dependent vertices. Thus, by the \Cref{0-ev-ld}, $\det(A_{G_H})=0$. 
\end{proof}
It is natural to inquire about the converse of \Cref{v=e}. The following theorem establishes that the converse is not true in general. That is,  for a hypergraph $H$ with the same number of vertices and hyperedges, the matrix $A_{G_H}$ is not necessarily non-singular. However, when $|V(H)|=|E(H)|$, and further, if we can ensure that $I_H$ is non-singular, then the non-singularity of $A_{G_H}$ is guaranteed.

     \begin{prop}\label{A-I}
    Let $H$ be a hypergraph with $|V(H)|=|E(H)|$. The adjacency matrix $A_{G_H}$ is non-singular if and only if $I_H$ is non-singular.
\end{prop}
\begin{proof}
If $A_{G_H}$ is non-singular, then by \Cref{0-ev-ld},  the collection $E(H)$ of all the hyperedges in $H$ is linearly independent, Thus, by the \Cref{I_H_lde}, $I_Hy\ne0$ for all non-zero vector $y:E(H)\to\mathbb{R}$. That is the square matrix $I_H$ is non-singular. Conversely, if the square matrix $I_H$ is non-singular, then its transpose $I_H^T$ is also non-singular. Therefore, by the \Cref{I_H_lde}, and the \Cref{Ix=0_ldv} ,  all the vertices  are linearly independent and all the hyperedges are linearly independent. Therefore, by \Cref{0-ev-ld}, $A_{G_H}$ is non-singular.


\end{proof}

If $|V(H)|\ne |E(H)|$, the matrix $A_{G_H}$ becomes singular, leading to $0$ is an eigenvalue of $A_{G_H}$. Subsequently, we aim to demonstrate that the multiplicity of the eigenvalue $0$ is at least $||V(H)| - |E(H)||$.

\begin{prop}\label{v-e-multiplicity}
    Let $H$ be a hypergraph. If $||V(H)|-|E(H)||\ne 0$, then $0$ is an eigenvalue of $A_{G_H}$ with multiplicity at least $||V(H)|-|E(H)||$.
\end{prop}
\begin{proof}   Let $H$ be a hypergraph with  $|V(H)|=n$, and $|E(H)|=m$.  
Given that $||V(H)|-|E(H)||\ne 0$, it implies that either $|E(H)|>|V(H)|$ or $|V(H)|>|E(H)|$. In the case where $|E(H)|>|V(H)|$, by \Cref{lin-ind-v=e}, there can be at most $n$ linearly independent hyperedges in $H$. Thus, given any maximal independent collection of hyperedges  $E\subseteq E(H)$ in $H$, there exists at least $m-n$ hyperedges, $e_1,\ldots,e_{m-n}\in E(H)\setminus E$. Since $E$ is a maximal collection of linearly independent hyperedges, for all $i=1,\ldots,m-n$ we have $\chi_{e_i}=\sum\limits_{e\in E}c^i_e\chi_{e}$  where $c_e^i\in\mathbb{R}$ for each $e\in E$. Thus, for all $i=1,\ldots,m-n$, we have $y_i:E(H)\to \mathbb{R}$ such that $y_i(e)=c^i_e$ for all $e\in E$,  $y_i(e_i)=-1$, and $y_i(e)=0$ for all $e(\ne e_i)\in E(H)\setminus E$.  Therefore, by the \Cref{I_H_lde}, $I_Hy_i=0$ for all $i=1,\ldots,m-n$. If $c_1y_1+\ldots+c_{m-n}y_{m-n}=0$ for $c_1,\ldots,c_{n-m}\in\mathbb{R}$, then $c_i=(c_1y_1+\ldots+c_{m-n}y_{m-n})(e_i)=0$ for all $i=1,\ldots,m-n$. Therefore, $y_1,\ldots,y_{m-n}$ are linearly independent vectors with $A_{G_H}\tilde y_i=0$ for all $i=1,\ldots,m-n$. Therefore, $0$ is an eigenvalue of $A_{G_H}$ with multiplicity at least $m-n=||V(H)|-|E(H)||$.
Similarly,  if $|V(H)|> |E(H)|$, then the number of linearly independent vertices in H can be at most $m$. Thus, given any maximal linearly independent collection of vertices $U$ such that there exist at least $m-n$ vertices in $v_1,\ldots,v_{n-m}\in V(H)\setminus U$.  Now proceeding similar to the previous case, these $n-m$ vertices lead us to the collection of linearly independent vectors $\{z_i:V(H)\to\mathbb{R}:i=1,\ldots,n-m\}$ such that  $A_{G_H}\tilde z_i=0$. Therefore, $0$ is an eigenvalue of $A_{G_H}$ with multiplicity at least $n-m=||V(H)|-|E(H)||$.

\end{proof}

The \Cref{0-ev-ld} leads us to the following Theorem for hypergraph with $|V(H)|=|E(H)|$.
\begin{thm}
    Let $H$ be a hypergraph with $|V(H)|=|E(H)|$. The following are equivalent.

\begin{enumerate}
\item The set of vertices $V(H)$ is a linearly independent set of vertices.
\item The set of hyperedges $E(H)$ is a linearly independent set of hyperedges.
\item The matrix $A_{G_H}$ is non-singular.
\end{enumerate}
\end{thm}
 \begin{proof}
 \textbf{ Proof of 1. implies 2.:} Since $V(H)$ is linearly independent and $|V(H)|=|E(H)|$, by the \Cref{lin-ind-v=e}, $E(H)$ is linearly independent.    
    
    \textbf{ Proof of 2. implies 3.:} Since $E(H)$ is linearly independent, by \Cref{I_H_lde}, the square matrix $I_H$ is non-singular. Therefore, by \Cref{A-I} $A_{G_H}$, is non-singular.

 \textbf{ Proof of 3. implies 1.:} Since the matrix $A_{G_H}$ is non-singular, the \Cref{0-ev-ld} implies the vertex set $V(H)$ is linearly independent.   
\end{proof}
 The above result indicates that in a hypergraph $H$ with $|V(H)|=|E(H)|$, the zero eigenvalues of $A_{G_H}$ indicate the existence of both the linearly dependent hyperedges and vertices. We illustrate this in the following example.
 \begin{exm}\rm
     1. Consider the hypergraph $H$ with $V(H)=\{1,2,3,4,5\}$, and $E(H)=\{e_1,e_2,e_3,e_4,e_5\}$ with $e_1=\{1,2,5\},e_2=\{2,3,5\},e_3=\{3,4,5\},e_4=\{1,4,5\},e_5=\{1,2\}$. Since $\chi_{e_1}-\chi_{e_2}+\chi_{e_3}=\chi_{e_4}$, $0$ is an eigenvalue of $A_{G_H}$. As, the \Cref{0-ev-ld} suggest, $s_1+s_3=s_2+s_4$, and thus $\{1,2,3,4\}$ is a linearly dependent set of vertices in $H$.

     2. Let $H$ be a hypergraph with $V(H)=\{1,2,3\ldots,n\}$, and $E(H)=\{e_1,e_2,\ldots,e_n\}$ is such that 
     $e_1=\{1\}$, $i\in e_i$, and $ e_i\subseteq \{1,2,\ldots,i\}$ for all $i=2,3,\ldots,n$. If we arrange the rows and columns of $I_H$ so that the $i$-th row and the $i$-th column of $I_H$ are indexed by the vertex $i$ and the hyperedge $e_i$ respectively, then $I_H$ is a lower-triangular matrix. Since $i\in e_i$ for all $i=1,2,\ldots,n$,  all the diagonal entries of the upper-triangular matrix $I_H$ are $1$. Therefore, $\det{I_H}=1$. Thus, $I_H$ is non-singular, that is, both $V(H)$ and $E(H)$ are linearly independent.

     3. Let $H$ be a hypergraph with $V(H)=\{1,2,\ldots,n\}$, and $E(H)=\{e_1,e_2,\ldots,e_n\}$ where $e_i=V(H)\setminus\{i\}$.  If we arrange the rows and columns of $I_H$ so that the $i$-th row and the $i$-th column of $I_H$ are indexed by the vertex $i$ and the hyperedge $e_i$ respectively, then all the non-diagonal entries of $I_H$ are $1$, and the diagonal entries are $0$. Thus, $I_H$ is a circulant matrix. Therefore,
     $$\det(I_H)=\prod\limits_{w:w^n=1}\sum\limits_{i=1}^{n-1}w^i=(-1)^{n-1}(n-1).$$ Consequently, $A_{G_H}$ is a non-singular matrix and the vertices and hyperedges of $H$ are linearly independent.
 \end{exm}
Some specific linearly dependent set of vertices and hyperedges corresponds to symmetry and redundancy in the structure of a hypergraph. Thus, our exploration in the subsequent section delves into the structural symmetry and redundancy corresponding to the nullspace of $A_{G_H}$.
\subsection{Some hypergraph substructures due to linearly dependent vertices and hyperedges.}\label{unit-parpart}
Thus far, our investigation has delved into the concepts of linearly dependent vertices and hyperedges as elucidated by the incident matrix. Presently, we focus on certain hypergraph substructures that lead to linearly dependent vertices and hyperedges. As the \Cref{0-ev-ld} suggest, we show that these hypergraph substructures are characterized by specific vectors in the null space of $A_{G_H}$. 
\subsubsection{Units in a hypergraph}\label{unit-a-g-h}
In a hypergraph $H$, units are the maximal collections of vertices with the same stars.
\begin{df}[Unit]\cite[Definition 3.1]{unit}
    Let $H$ be a hypergraph. Consider the equivalence relation \[\mathcal{R}_u(H) = \{(u, v) \in V(H) \times V(H) : E_u(H) = E_v(H)\}.\] Each equivalence class under \(\mathcal{R}_u(H)\) is referred to as a \emph{unit}. For every unit \(W_E (\subseteq V(H))\), there exists a corresponding collection \(E \subseteq E(H)\) such that \(E_v(H)=E\) for all \(v \in W_E\). This collection $E$ is called the \emph{generator} of the unit $W_E$.
\end{df}
 We denote the complete collection of units in $H$ as $\mathfrak{U}(H)$. Given a hypergraph $H$, the \emph{unit-contraction} of $H$ is a hypergraph $H/\mathcal{R}_u(H)$ with \[V(H/\mathcal{R}_u(H))=\mathfrak{U}(H),\]
and \[E(H/\mathcal{R}_u(H))=\{\Tilde{e}=\{W_{E_v(H)}:v\in e\}: e\in E(H)\}.\] 
For any $e\in E(H)$, $\Tilde{e}$ is a set, and if $E_u(H)=E=E_v(H)$ for two $u,v\in e$, then $\Tilde{e}$ contain the unit $W_E$ containing $u,v$. However, to avoid possible confusion, it is important to clarify that being a set, $\Tilde{e}$ contains $W_E$ just once and does not contain two distinct instances of $W_E$ for both $u$ and $v$ individually. 
\begin{exm}\label{ex-unit}\rm
    \begin{figure}[H]
        \centering
        \begin{subfigure}{0.49\textwidth}
        \begin{tikzpicture}[scale=0.65]
		\node [style=none] (11) at (4, 5) {};
		\node [style=none] (12) at (6, 4.25) {};
		\node [style=none] (13) at (5.5, 2.25) {};
		\node [style=none] (14) at (0.25, 0.25) {};
		\node [style=none] (15) at (2.25, -2.25) {};
		\node [style=none] (16) at (3.5, -2) {};
		\node [style=none] (17) at (-1.75, 4.25) {};
		\node [style=none] (18) at (-0.25, 6.5) {};
		\node [style=none] (19) at (1.75, 4.75) {};
		\node [style=none] (20) at (0.25, 0.5) {};
		\node [style=none] (21) at (3, -1.25) {};
		\node [style=none] (22) at (1.75, 1.75) {};
		\node [style=none] (24) at (-4.75, 2.75) {};
		\node [style=none] (25) at (-3.75, 4) {};
		\node [style=none] (26) at (-2.25, 2.25) {};
		\node [style=none] (27) at (-4, 0.75) {};
		\node [style=none] (28) at (-3, -0.25) {};
		\node [style=none] (29) at (0.75, 3.5) {};
		\node [style=none] (30) at (-0.25, 2.25) {};
		\node [style=none] (31) at (0.75, 5) {};
		\node [style=none] (32) at (5, 4) {};
		\node [style=none] (33) at (2.75, 1.75) {};
		\node [style=none] (34) at (3.75, 0.75) {};
		\node [style=none] (35) at (4.75, 1.75) {};
		\node [style=none] (36) at (-2.25, 5.75) {};
		\node [style=none] (37) at (-0.75, 6.5) {};
		\node [style=none] (38) at (0.25, 5.5) {};
		\node [style=none] (39) at (-4.25, 2.5) {};
		\node [style=none] (40) at (-2.75, 0.25) {};
		\node [style=none] (41) at (-1.5, 1) {};
		\node [style=none] (42) at (0.5, 0) {};
		\node [style=none] (43) at (1.75, 1.5) {};
		\node [style=none] (44) at (3, -1.25) {};
		\node [style=none] (45) at (3, 2.75) {};
		\node [style=none] (46) at (4.25, 4.25) {};
		\node [style=none] (47) at (3.75, 1) {};
		\node [style=none] (48) at (-4, 2) {};
		\node [style=none] (49) at (-3.25, 3.25) {};
		\node [style=none] (50) at (-2, 0.5) {};
		\node [style=none] (51) at (-1.5, 5.5) {};
		\node [style=none] (52) at (-0.25, 6.25) {};
		\node [style=none] (53) at (-1, 3.75) {};
		\node [style=none] (56) at (0.25, 3.25) {};
		\node [style=none] (57) at (0, 2.25) {};
		\node [style=none] (58) at (1.25, 5) {};
		\node [style=none] (59) at (1, 3.25) {};
		\node [style=none] (67) at (0.75, 2.25) {};
		\node [style=none] (68) at (1, -1.25) {};
		\node [style=none] (69) at (3.75, 2.75) {};
		\node [style=none] (70) at (2, 4.25) {};
		\node [style=none] (71) at (1, 2.75) {};
		\draw (14.center)
			 to [in=-150, out=45, looseness=0.75] (11.center)
			 to [bend left=45] (12.center)
			 to [bend left=45] (13.center)
			 to [in=60, out=-165, looseness=1.50] (16.center)
			 to [bend left=45] (15.center)
			 to [bend left=45] cycle;
		\draw (20.center)
			 to [in=-30, out=135] (17.center)
			 to [bend left=45] (18.center)
			 to [bend left=45] (19.center)
			 to [in=135, out=-75, looseness=0.75] (22.center)
			 to [bend left=45] (21.center)
			 to [bend left=45] cycle;
		\draw (27.center)
			 to (24.center)
			 to [bend left=45] (25.center)
			 to [in=165, out=0, looseness=1.25] (26.center)
			 to [in=180, out=0] (29.center)
			 to [bend left=45] (28.center)
			 to [bend left=45] cycle;
		\draw (33.center)
			 to [in=-45, out=165, looseness=1.25] (30.center)
			 to [bend left=45] (31.center)
			 to [bend left=45] (32.center)
			 to [in=135, out=-75, looseness=1.25] (35.center)
			 to [bend left=45] (34.center)
			 to [bend left=45] cycle;
		\draw (39.center)
			 to [in=-90, out=60] (36.center)
			 to [bend left=45] (37.center)
			 to [bend left=45] (38.center)
			 to [in=90, out=-135] (41.center)
			 to [bend left=45] (40.center)
			 to [bend left=45] cycle;
		\draw [style=new edge style 1,fill=gray!30!white] (42.center)
			 to [bend right=45] (44.center)
			 to [in=-30, out=90] (43.center)
			 to [bend right=45, looseness=0.50] cycle;
		\draw [style=new edge style 1,fill=gray!30!white] (45.center)
			 to [bend right=45] (47.center)
			 to [in=-30, out=90] (46.center)
			 to [bend right=45, looseness=0.50] cycle;
		\draw [style=new edge style 1,fill=gray!30!white] (48.center)
			 to [bend right=45] (50.center)
			 to [in=-75, out=90] (49.center)
			 to [bend right=45, looseness=0.50] cycle;
		\draw [style=new edge style 1,fill=gray!30!white] (51.center)
			 to [bend right=45] (53.center)
			 to [in=-30, out=90] (52.center)
			 to [bend right=45, looseness=0.50] cycle;
		\draw [style=new edge style 1,fill=gray!30!white] (57.center)
			 to [bend left=105, looseness=2.00] (56.center)
			 to [bend left=105, looseness=2.00] cycle;
		\draw [style=new edge style 1,fill=gray!30!white] (59.center)
			 to [bend left=105, looseness=2.00] (58.center)
			 to [in=-15, out=0] cycle;
    	\node [style=new style 0] (0) at (-0.25, 5.75) {\tiny 1};
		\node [style=new style 0] (1) at (-1.25, 4.5) {\tiny 2};
		\node [style=new style 0] (2) at (1, 4) {\tiny  11};
		\node [style=new style 0] (3) at (-3.5, 2.5) {\tiny  3};
		\node [style=new style 0] (4) at (-2.5, 0.75) {\tiny  4};
		\node [style=new style 0] (5) at (1.75, 1) {\tiny  7};
		\node [style=new style 0] (6) at (1, 0) {\tiny  5};
		\node [style=new style 0] (7) at (2.5, 0) {\tiny 6};
		\node [style=new style 0] (8) at (4.25, 3.75) {\tiny  9};
		\node [style=new style 0] (9) at (3.5, 1.5) {\tiny 8};
            \node [style=new style 0] (23) at (0, 2.75) {\tiny  10};
            	\node [style=none] (60) at (1, 6.75) {\tiny $e_1$};
		\node [style=none] (61) at (-2.75, 5.5) {\tiny  $e_2$};
		\node [style=none] (62) at (-3.75, -0.5) {\tiny $e_3$};
		\node [style=none] (63) at (0, -1) {\tiny $e_4$};
		\node [style=none] (64) at (3, 5.75) {\tiny  $e_5$};
		\node [style=none] (65) at (-0.75, 5.25) {\tiny  $W_{E_1}$};
		\node [style=none] (66) at (-3, 1.5) {\tiny  $W_{E_2}$};
  \node [style=none] (72) at (2, -0.75) {\tiny  $W_{E_3}$};
		\node [style=none] (73) at (3.75, 2.75) {\tiny  $W_{E_4}$};
		\node [style=none] (74) at (1, 4.75) {\tiny  $W_{E_5}$};
		\node [style=none] (75) at (1.05, 2.5) {\tiny  $W_{E_6}$};
		\node [style=none] (76) at (0.25, -3.75) {$H$};
\end{tikzpicture}
         \caption{A hypergraph $H$ wherein units are identified within the shaded regions.}
         \label{fig:hyp-unit}
        \end{subfigure}
        \begin{subfigure}{0.5\textwidth}
        \begin{tikzpicture}[scale=0.6]
		\node [style=none] (11) at (4, 5) {};
		\node [style=none] (12) at (6, 4.25) {};
		\node [style=none] (13) at (5.5, 2.25) {};
		\node [style=none] (14) at (0.25, 0.25) {};
		\node [style=none] (15) at (2.25, -2.25) {};
		\node [style=none] (16) at (3.5, -2) {};
		\node [style=none] (17) at (-1.75, 4.25) {};
		\node [style=none] (18) at (-0.25, 6.5) {};
		\node [style=none] (19) at (1.75, 4.75) {};
		\node [style=none] (20) at (0.25, 0.5) {};
		\node [style=none] (21) at (3, -1.25) {};
		\node [style=none] (22) at (1.75, 1.75) {};
		\node [style=none] (24) at (-4.75, 2.75) {};
		\node [style=none] (25) at (-3.75, 4) {};
		\node [style=none] (26) at (-2.25, 2.25) {};
		\node [style=none] (27) at (-4, 0.75) {};
		\node [style=none] (28) at (-3, -0.25) {};
		\node [style=none] (29) at (0.75, 3.5) {};
		\node [style=none] (30) at (-0.25, 2.25) {};
		\node [style=none] (31) at (0.75, 5) {};
		\node [style=none] (32) at (5, 4) {};
		\node [style=none] (33) at (2.75, 1.75) {};
		\node [style=none] (34) at (3.75, 0.75) {};
		\node [style=none] (35) at (4.75, 1.75) {};
		\node [style=none] (36) at (-2.25, 5.75) {};
		\node [style=none] (37) at (-0.75, 6.5) {};
		\node [style=none] (38) at (0.25, 5.5) {};
		\node [style=none] (39) at (-4.25, 2.5) {};
		\node [style=none] (40) at (-2.75, 0.25) {};
		\node [style=none] (41) at (-1.5, 1) {};
		\node [style=none] (42) at (0.5, 0) {};
		\node [style=none] (43) at (1.75, 1.5) {};
		\node [style=none] (44) at (3, -1.25) {};
		\node [style=none] (45) at (3, 2.75) {};
		\node [style=none] (46) at (4.25, 4.25) {};
		\node [style=none] (47) at (3.75, 1) {};
		\node [style=none] (48) at (-4, 2) {};
		\node [style=none] (49) at (-3.25, 3.25) {};
		\node [style=none] (50) at (-2, 0.5) {};
		\node [style=none] (51) at (-1.5, 5.5) {};
		\node [style=none] (52) at (-0.25, 6.25) {};
		\node [style=none] (53) at (-1, 3.75) {};
		\node [style=none] (56) at (0.25, 3.25) {};
		\node [style=none] (57) at (0, 2.25) {};
		\node [style=none] (58) at (1.25, 5) {};
		\node [style=none] (59) at (1, 3.25) {};
		\node [style=none] (67) at (0.75, 2.25) {};
		\node [style=none] (68) at (1, -1.25) {};
		\node [style=none] (69) at (3.75, 2.75) {};
		\node [style=none] (70) at (2, 4.25) {};
		\node [style=none] (71) at (1, 2.75) {};
		\draw (14.center)
			 to [in=-150, out=45, looseness=0.75] (11.center)
			 to [bend left=45] (12.center)
			 to [bend left=45] (13.center)
			 to [in=60, out=-165, looseness=1.50] (16.center)
			 to [bend left=45] (15.center)
			 to [bend left=45] cycle;
		\draw (20.center)
			 to [in=-30, out=135] (17.center)
			 to [bend left=45] (18.center)
			 to [bend left=45] (19.center)
			 to [in=135, out=-75, looseness=0.75] (22.center)
			 to [bend left=45] (21.center)
			 to [bend left=45] cycle;
		\draw (27.center)
			 to (24.center)
			 to [bend left=45] (25.center)
			 to [in=165, out=0, looseness=1.25] (26.center)
			 to [in=180, out=0] (29.center)
			 to [bend left=45] (28.center)
			 to [bend left=45] cycle;
		\draw (33.center)
			 to [in=-45, out=165, looseness=1.25] (30.center)
			 to [bend left=45] (31.center)
			 to [bend left=45] (32.center)
			 to [in=135, out=-75, looseness=1.25] (35.center)
			 to [bend left=45] (34.center)
			 to [bend left=45] cycle;
		\draw (39.center)
			 to [in=-90, out=60] (36.center)
			 to [bend left=45] (37.center)
			 to [bend left=45] (38.center)
			 to [in=90, out=-135] (41.center)
			 to [bend left=45] (40.center)
			 to [bend left=45] cycle;
            \node [rectangle,draw] (23) at (0, 2.75) {};
            	\node [style=none] (60) at (1, 6.75) {\tiny $\tilde e_1$};
		\node [style=none] (61) at (-2.75, 5.5) {\tiny  $\tilde e_2$};
		\node [style=none] (62) at (-3.75, -0.5) {\tiny $\tilde e_3$};
		\node [style=none] (63) at (0, -1) {\tiny $\tilde e_4$};
		\node [style=none] (64) at (3, 5.75) {\tiny  $\tilde e_5$};
		\node [rectangle,draw] (65) at (-0.75, 5.25) {\tiny  $W_{E_1}$};
		\node [rectangle,draw] (66) at (-3, 1.5) {\tiny  $W_{E_2}$};
  \node [rectangle,draw] (72) at (2, -0.5) {\tiny  $W_{E_3}$};
		\node [rectangle,draw] (73) at (3.75, 2.75) {  \tiny$W_{E_4}$};
		\node [rectangle,draw] (74) at (1, 4.5) {\tiny  $W_{E_5}$};
		\node [none] (75) at (1.05, 2.5) {\tiny  $W_{E_6}$};
		\node [style=none] (76) at (0.25, -3.75) {$H/\mathcal{R}_u(H)$};
\end{tikzpicture}
         \caption{Units of $H$ become vertices in $H/\mathcal{R}_u(H)$, the unit-contraction of $H$. 
         }
         \label{fig:unit-contraction}
        \end{subfigure}
        \caption{Units in a hypergraphs and unit contraction}
        \label{fig:unit and contraction}
    \end{figure}
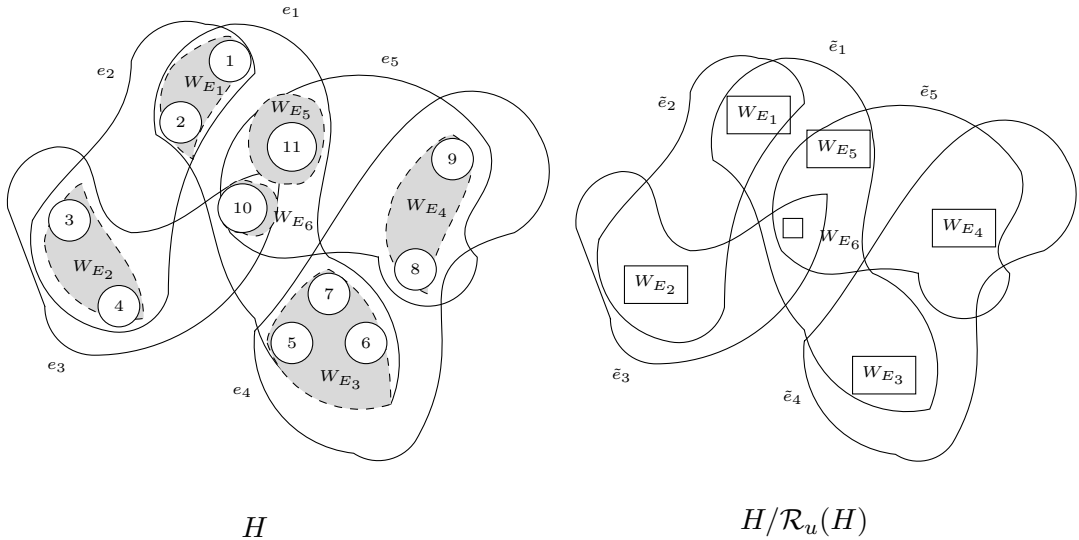
   Consider the hypergraph $H$  with $V(H)=\{1,2,\ldots,10,11\}$, and $E(H)=\{e_1,e_2,e_3,e_4,e_5\}$ (see \Cref{fig:hyp-unit}), where $ e_1=\{1,2,5,6,7,10,11,\}$, $e_2=\{1,2,3,4\}, e_3=\{3,4,10\}$,
       $e_4=\{5,6,7,8,9\}$, $e_5=\{8,9,10,11\}$. The units in $H$ are$ W_{E_1}=\{1,2\}$, $W_{E_2}=\{3,4\}$, $W_{E_3}=\{5,6,7\}$, $W_{E_4}=\{8,9\}$, $W_{E_5}=\{11\}$, $W_{E_6}=\{10\}$. The corresponding generating sets are $E_1=\{e_1,e_2\}$, $E_2=\{e_2,e_3\}$, $E_3=\{e_1,e_4\}$, $E_4=\{e_4,e_5\}$, $E_5=\{e_1,e_5\}$, $E_6=\{e_1,e_3,e_5\}$.
\end{exm}
A hypergraph $H$ is called \emph{non-contractible} if each unit is a singleton set. If $H$ is non-contractible then $H$ is isomorphic to $H/\mathcal{R}_u(H)$. 
\begin{prop}
    Let $H$ be a hypergraph. If $A_{G_H}$ is non-singular, then $H$ is non-contractible.
\end{prop}
\begin{proof}
    Assuming the contrary, let $H$ not be non-contractible. In such a case, a unit $W_E\in\mathfrak{U}(H)$ with $u,v(\ne u)\in W_E$. Consequently, the rows corresponding to $u$ and $v$ in $I_H$ are identical, leading to the linearly dependent subset of vertices  $\{u,v\}$.  Consequently, in accordance with \Cref{0-ev-ld}, it follows that $\det(A_{G_H})=0$, which is a contradiction to the fact that $A_{G_H}$ is non-singular. Hence, we conclude that $H$ is non-contractible.
\end{proof}
The aforementioned Proposition reveals that every unit of cardinality exceeding $1$ is associated with a $0$ eigenvalue of the matrix $A_{G_H}$. In the subsequent Theorem, we provide the corresponding eigenspace. The converse does not hold true; the presence of a $0$ eigenvalue in $A_{G_H}$ does not necessarily ensure the existence of a unit in $H$. A unit with cardinality $\ge 2$ is a linearly dependent set of vertices. Still, other linearly dependent sets of vertices (or hyperedges), which are not units, can be the reason for a $0$ eigenvalue of $A_{G_H}$. In our subsequent Theorem, we demonstrate that a $0$ eigenvalue in $A_{G_H}$, combined with a specific condition on the nullspace of $A_{G_H}$, guarantees the existence of a unit in $H$. However, before delving into that, we need to introduce the following vector space.
 Let $A$ be a non-empty set, and $\mathbb{R}^A$ be the collection of all real-valued functions on $A$. For any finite $B(\ne\varnothing)\subseteq A$, we define 
 \[S_A=\{f\in\mathbb{R}^A:f(a)=0 \text{~if~}a\in A\setminus B, \text{~and~}\sum\limits_{b\in B}f(b)=0 \}.\]
Given any finite subset $W\subseteq V(H)$ with $|W|=k>1$, we can enumerate $W=\{v_0,v_1,\ldots,v_{k-1}\}$. For two distinct $u,v\in V(H)$, we define a function
\(x_{uv}:V(H)\to\mathbb{R} \text{~as}\) 
\[x_{uv}(w)=\begin{cases}
    -1&\text{~if~}w=u,\\
    \phantom{-}1&\text{~if~}w=u,\\
     \phantom{-}0&\text{~otherwise}.
\end{cases}\]
The collection $\{x_{v_0v_i}:i=1,2,\ldots,k-1\}$ of linearly independent vectors forms a basis of the subspace $S_{W}$ of $\mathbb{R}^{V(H)}$. Therefore, the dimension of $S_W$ is $|W|-1$. Since $V(H)\subseteq V(G_H)$, the vector space $\mathbb{R}^{V(H)}$ can be embedded in $\mathbb{R}^{V(G_H)}$ by the embedding map
\(x\mapsto \tilde{x} \text{~for all~}x\in \mathbb{R}^{V(H)},\)
where the embedding of $x$, $\tilde{x}:V(G_H)\to\mathbb{R}$ is such that $\tilde{x}(v)= x(v)$ if $v\in V(H)$, and $\tilde{x}(v)=0$ for all $v\in V(G_H)\setminus V(H)$.
Under this embedding, the subspace $S_{W}$ of $\mathbb{R}^{V(H)}$ is also embedded as subspace $\tilde S_{W}$ of $\mathbb{R}^{V(G_H)}$. Thus, $\tilde S_{W}$ is the subspace generated by the collection $\{\tilde x_{v_0v_i}:V(G_H)\to\mathbb{R}:i=1,2,\ldots,k-1\}$.
For any $W\subset V(H)$, we say $W$ is \emph{maximal} such that $S_W$ (or $\tilde S_W$) satisfy a property $\mathcal P$ if $S_W$ (or $\tilde S_W$) satisfy the property $\mathcal P$, and there is no $W'$ with $W\subset W'\subseteq V(H)$ such that $T_W$ ( $\tilde T_W$) satisfy the property $\mathcal P$.
\begin{thm}\label{null-ag}
     Let $H$ be a hypergraph, $W\subseteq V(H)$, and $|W|\ge 2$. The set  $W$ is a unit in $H$ if and only if $0$ is an eigenvalue of $A_{G_H}$ with multiplicity at least $|W|-1$, and $W$ is a maximal subset of $V(H)$ such that $\tilde S_W$ is a subspace of the nullspace of $A_{G_H}$.
\end{thm}
\begin{proof}
    Let $W$ be a unit with the generating set $E$, and $W=W_E=\{v_0,\ldots,v_k\}$. Therefore, $s_{v_i}-s_{v_0}=0$ for any $i=1,\ldots,k$. Consequently, For all $e\in E(H)$,
    $$(I_H^Tx_{v_0v_i})(e)=\sum\limits_{v\in V(H)}i_{ve}x_{v_0v_i}(v)=s_{v_i}(e)-s_{v_0}(e)=0\text{~for all~}i=1,\ldots,k.$$
    Therefore, $I_H^Tx_{v_0v_i}=0$, and $A_{G_H}\tilde x_{v_0v_i}=0$. Since $\{\tilde x_{v_0v_i}:i=1,\ldots,k\}$ are linearly independent sets of vectors, the multiplicity of the $0$ eigenvalue of $A_{G_H}$ is at least $k=|W|-1$, and $\tilde S_W$ is a subspace of the nullspace of $A_{G_H}$. 
    If possible, let $W$ be not a maximal set with the property  $\tilde S_W$ is a subspace of the nullspace of $A_{G_H}$. In that case, there exists $W'$ with $W\subsetneq W'\subseteq V(H)$ such that $\tilde S'_W$ is a subspace of the nullspace of $A_{G_H}$. Thus, for all $u\in W'\setminus W$ we have $\tilde x_{v_0u}\in \tilde S'_W$ and $A_{G_H}\tilde x_{v_0u}=0$. Therefore, $I_H^Tx_{v_0u}=0$, and  
    $s_{u}(e)-s_{v_0}(e)= (I_H^Tx_{v_0u})(e)=0$ for any $e\in E(H)$. That is, $E_u(H)=E_{v_0}(H)=E_{v_i}(H)$ for all $i=1,\ldots,k$. Therefore, $W'$ is a collection of vertices having the same star with $ W\subsetneq W'$. Being a unit, $W$ is a maximal collection of vertices which have the same stars, but since $ W\subsetneq W'$, this is a contradiction to the maximality of $W$ as a collection of vertices which have the same star. Therefore, our assumption is wrong, and $W$ is a maximal set with the property  $\tilde S_W$ as a subspace of the nullspace of $A_{G_H}$.

    For the converse part, suppose that $0$ is an eigenvalue of $A_{G_H}$ with multiplicity at least $|W|-1$, and $W$ is a maximal subset of $V(H)$ such that $\tilde S_W$ is a subspace of the nullspace of $A_{G_H}$. Thus, for all $i=1,\ldots,k$, we have $A_{G_H}\tilde x_{v_0v_i}=0$, and consequently,  $s_{v_i}(e)-s_{v_0}(e)= (I_H^Tx_{v_0v_i})(e)=0$ for all $e\in E(H)$. Therefore, $E_{v_0}(H)=E_{v_i}(H)$ for all $i=1,\ldots,k$. That is, $W$ is a maximal collection of vertices which have the same stars. Therefore, $W$ is a unit.
\end{proof}

For any $x:\mathfrak{U}(H)\to\mathbb{R}$, we define $\hat x:V(H)\to \mathbb{R}$ as $\hat x(v)=\frac{1}{|W_E|}x(W_E)$, where $v\in W_E$. For any $y:E(H/\mathcal{R}_u(H))\to\mathbb{R}$, we define $ y':E(H)\to\mathbb{R}$ as $y'(e)=y(\tilde e)$. For any $z:V(G_{H/\mathcal{R}_u(H)})\to\mathbb{R}$, there exists $z_1:V(H/\mathcal{R}_u(H))\to\mathbb{R}$, and $z_2:E(H/\mathcal{R}_u(H))\to\mathbb{R}$ such that $z_1=z|_{V(H/\mathcal{R}_u(H))}$, the restriction of $z$ on $V(H/\mathcal{R}_u(H))$, and $z_2=z|_{E(H/\mathcal{R}_u(H))}$, the restriction of $z$ on $E(H/\mathcal{R}_u(H))$
We define $\bar z:V(G_H)\to\mathbb{R}$ as 
$$\bar z(w)=
\begin{cases}
   \hat z_1(w)&\text{~if~}w\in V(H/\mathcal{R}_u(H))\\
    z_2'(w)&\text{~if~}w\in E(H/\mathcal{R}_u(H)).
\end{cases}$$
Units are specific linearly dependent sets of vertices in hypergraphs. Unit contraction removes the linear dependency due to units. Thus, the nullity of $A_{G_H}$ is greater than or equal to the nullity of $A_{G_{H/\mathcal{R}_u(H)}}$, and the equality occurs when $H$ is non-contractible. Thus, we have the following result.  
\begin{thm}\label{contraction-null}
    Let $H$ be a hypergraph. For any vector $z:V(G_{H/\mathcal{R}_u(H)})\to\mathbb{R}$ in the nullspace of $A_{G_{H/\mathcal{R}_u(H)}}$, the vector $\bar z:V(G_H)\to\mathbb{R}$ lies in the nullspace of $A_{G_H}$.
\end{thm}
\begin{proof}
    Suppose that $z:V(G_{H/\mathcal{R}_u(H)})\to\mathbb{R}$ is such that $A_{G_{H/\mathcal{R}_u(H)}}z=0$. It is enough to prove $A_{G_H}\bar z=0$. For $z:V(G_{H/\mathcal{R}_u(H)})\to\mathbb{R}$, suppose that $z_1=z|_{V(H/\mathcal{R}_u(H))}$, the restriction of $z$ on $V(H/\mathcal{R}_u(H))$, and $z_2=z|_{E(H/\mathcal{R}_u(H))}$, the restriction of $z$ on $E(H/\mathcal{R}_u(H))$.
Therefore, $A_{G_{H/\mathcal{R}_u(H)}}z=0$ leads us to $I_{H/\mathcal{R}_u(H)}z_2=0$, and $I_{H/\mathcal{R}_u(H)}^Tz_1=0$. Since $I_{H/\mathcal{R}_u(H)}z_2=0$, for any $W_E\in V(H/\mathcal{R}_u(H))$, 
$$ 0=(I_{H/\mathcal{R}_u(H)}z_2)(W_E)=\sum\limits_{\tilde e: W_E\in \tilde e}z_2(\tilde e).$$

\noindent Consequently, for any $v\in W_E$ we have $(I_Hz_2')(v)=\sum\limits_{e\in E_v(H)}z_2'(e)=\sum\limits_{\tilde e: W_E\in \tilde e}z_2(\tilde e)=0$. Thus, $I_{H/\mathcal{R}_u(H)}z_2=0$ leads us to $I_Hz_2'=0 $.

 \noindent Similarly, $I_{H/\mathcal{R}_u(H)}^Tz_1=0$ yields for any $\tilde e\in E(H/\mathcal{R}_u(H))$,
$$0=(I_{H/\mathcal{R}_u(H)}^Tz_1)(\tilde e)=\sum\limits_{W_E\in \tilde e}z_1(W_E)=\sum\limits_{W_E\in \tilde e}\sum\limits_{v\in W_E}\tilde z_1(v).$$
Therefore, for any $e\in E(H)$, we have $(I_H^T\hat z_1)(e)=\sum\limits_{v\in e}\hat z_1(v)=\sum\limits_{W_E\in \tilde e}\sum\limits_{v\in W_E}\hat z_1(v)=0$, and $I_H^T\hat z_1=0$. 

\noindent Consequently, $A_{G_{H/\mathcal{R}_u(H)}}z=0$ implies $I_Hz_2'=0 $, and $I_H^T\tilde z_1=0$. Since $\bar z:V(G_H)\to\mathbb{R}$ is such that 
$$\bar z(w)=
\begin{cases}
   \tilde z_1(w)&\text{~if~}w\in V(H/\mathcal{R}_u(H))\\
    z_2'(w)&\text{~if~}w\in E(H/\mathcal{R}_u(H)),
\end{cases}$$
we have \begin{align*}
    (A_{G_H}\bar z)(w)&=\begin{cases}
    (I_Hz_2')(w)&\text{~if~}w\in V(H/\mathcal{R}_u(H))\\
   (I_H^T\tilde z_1)(w)&\text{~if~}w\in E(H/\mathcal{R}_u(H))
\end{cases}\\
&=0.
\end{align*}
\end{proof}
The \Cref{contraction-null} shows that unit contraction reduces the nullity of the adjacency of the incidence graph.
In the following result, we show that if a hypergraph $H$ is such that there are no other linearly dependent structures other than the units of cardinality more than $1$, then the nullity of $A_{G_{H/\mathcal{R}_u(H)}}$ is guaranteed to be reduced to zero. 
\begin{thm}\label{unit-sigularity removal}
    Let $H$ be a hypergraph. If the nullity of $A_{G_H}$ is $\sum\limits_{W_E\in \mathfrak{U}(H)}(|W_E|-1)$, then the the adjacency matrix $A_{G_{H/\mathcal{R}_u(H)}}$ is non-singular. 
\end{thm}
\begin{proof}
    If possible let $A_{G_{H/\mathcal{R}_u(H)}}$ is singular. In that case there is a non zero vector $z:V(G_{H/\mathcal{R}_u(H)})\to\mathbb{R}$ such that $A_{G_{H/\mathcal{R}_u(H)}}z=0$. By \Cref{contraction-null}, $A_{G_H}\bar z=0$. For any $W_E=\{v_0,v_1,\ldots,v_k\}\in \mathfrak{U}(H)$, since $\{\tilde x_{v_0v_i}:i=1,2,\ldots,k\}$ is a basis of $\tilde S_{W_E}$, by \Cref{null-ag}, $\tilde S_{W_E}$ is a subspace of the nullspace of $A_{G_H}$. By definition, $\bar z$ is constant on $W_E$. Consequently, $\sum\limits_{w\in V(G_H)}\bar z(w)x(w)=0$ for all $x\in \tilde S_{W_E}$. the vector $\bar z$ is orthogonal with all $x\in \tilde S_{W_E}$ for all $W_E\in\mathfrak{U}(H)$. Since dimension of $\tilde S_{W_E}$ is $|W_E|-1$, the nullity of $A_{G_H}$ is at least $\left(\sum\limits_{W_E\in \mathfrak{U}(H)}(|W_E|-1)\right)+1$, a contradiction. Therefore, our assumption is wrong and $A_{G_{H/\mathcal{R}_u(H)}}$ is non-singular.
\end{proof}
\begin{cor}
    Let $H$ be a hypergraph. If the nullity of $A_{G_H}$ is $\sum\limits_{W_E\in \mathfrak{U}(H)}(|W_E|-1)$ then the number of unit in $H$, $|\mathfrak{U}(H)|=|E(H)|$, the number of hyperedges in $H$.
\end{cor}

\subsubsection{Equal partition of hyperedges}
Given two disjoint subsets $U=\{u_1,u_2,\ldots,u_p\}$, and $V=\{v_1,v_2,\ldots,v_q\}$ of the vertex set $V(H)$ of a hypergraph $H$, if 
$|U\cap e|=|V\cap e| $ for all $e\in E(H)$, then the pair $U$, and $V$ is called an \emph{equal partition of hyperedges} in $H$. An equal partition of hyperedges $U, V$ gives a linearly dependent set of vertices $W=U\cup V$. 

The corresponding coefficient vector $c_W:V(H)\to\mathbb{R}$ is defined by $c_W=\chi_{U}-\chi_V$. Thus, $\sum\limits_{v\in V(H)}c_W(v)s_v=0$. Consequently, the \Cref{Ix=0_ldv} leads to the following Theorem.
\begin{thm}\label{par-part}
    Let $H$ be a hypergraph. Two disjoint subsets $U=\{u_1,u_2,\ldots,u_p\}$, and $V=\{v_1,v_2,\ldots,v_q\}$ of the vertex set $V(H)$ forms an equal partition of hyperedges in $H$ if and only if $A_{G_H}\tilde x=0$, where $x=\chi_{U}-\chi_V$.
\end{thm}
\begin{proof}
Let $U=\{u_1,u_2,\ldots,u_p\}$, and $V=\{v_1,v_2,\ldots,v_q\}$ form an equal partition of hyperedges in $H$.
Consider the function $y=(\sum_{i=1}^ps_{u_i})-(\sum_{j=1}^qs_{v_i}):E(H)\to \mathbb{R}$. Since $|e\cap U|=|e\cap V|$ we have $y(e)=|e\cap U|-|e\cap V|=0$ for all $e\in E(H)$. Therefore, $(\sum_{i=1}^ps_{u_i})-(\sum_{j=1}^qs_{v_i})=0$, and $W=U\cup V$ is a linearly dependent set of vertices, and $x=\chi_{U}-\chi_V$ is a coefficient vector of $W$. That is, $ \sum\limits_{v\in W}x(v)s_v=0$, where $x=\chi_{U}-\chi_V$. Since $supp(x)=W$, $$(I_H^Tx)(e)=\sum\limits_{v\in V(H)}i_{ve}x(v)=\sum\limits_{v\in V(H)}x(v)s_v(e)=\sum\limits_{v\in W}x(v)s_v(e)=0\text{~for all~}e\in E(H).$$  
Therefore, $A_{G_H}\tilde x=0$.

    Conversely, if  $A_{G_H}\tilde x=0$, then $I_H^Tx=0$. Therefore, $\sum\limits_{v\in e }(\chi_{U}-\chi_V)(v)=(I_H^Tx)(e)=0$ for all $e\in E(H)$. Therefore, $|e\cap U|=\sum\limits_{v\in e }\chi_{U}(v)=\sum\limits_{u\in e }\chi_{V}(u)=|e\cap V|$ for all $e\in E(H) $.
\end{proof}

 Let $H$ be a hypergraph, and $v_0,v_1,\ldots,v_k\in V(H)$ be such that $E_{v_0}(H)=E_{v_1}(H)\cup E_{v_2}(H)\cup\ldots\cup E_{v_k}(H)$ with $E_{v_i}(H)\cap E_{v_j}(H)=\emptyset$ for all $1\le i<j\le k$. We say $v_1,\ldots,v_k$ forms a \emph{partition of the star} of $v_0$. 
 If $U=\{v_0\}$, and $V=\{v_1,\ldots,v_k\}$, then since $E_{v_i}(H)\cap E_{v_j}(H)=\emptyset$ for all $1\le i<j\le k$, the cardinality $|e\cap U|=|e\cap V|$ for all $e\in E(H)$. 
 Thus, we have the following Corollary of the \Cref{par-part}.
 \begin{cor}
     \label{th-part-ver}
    Let $H$ be a hypergraph and $v_0,v_1,\ldots,v_k\in V(H)$ are distinct vertices in $H$. The vertices $v_1,\ldots,v_k\in V(H)$ forms a partition of the star of $v_0\in V(H)$ if and only if $A_{G_H}x=0$ where $x:V(H)\to\mathbb{R}$ is such that  $supp(x)=\{v_0,v_1,\ldots,v_k\}\subseteq V(H)$, and $-x(v_0)=x(v_1)=\ldots=x(v_k)$.
 \end{cor}
 \begin{proof}
     Let $U=\{v_0\}$, and $V=\{v_1,\ldots,v_k\}$. Since $E_{v_0}(H)=E_{v_1}(H)\cup E_{v_2}(H)\cup\ldots\cup E_{v_k}(H)$ with $E_{v_i}(H)\cap E_{v_j}(H)=\emptyset$ for all $1\le i<j\le k$, If $e\in E_{v_0}(H)$, then there exists exactly one $v_i\in \{v_1,\ldots,v_k\}$ such that $e\in E_{v_i}(H)$. That is, $|e\cap U|=1=|e\cap V|$. If $e\in E(H)\setminus E_{v_0}(H)$, then $E_{v_0}(H)=E_{v_1}(H)\cup E_{v_2}(H)\cup\ldots\cup E_{v_k}(H)$ leads to $e\notin E_{v_i}(H)$ for all $i=1,\ldots,k$. Therefore, $|e\cap U|=0=|e\cap V|$. Consequently, $|e\cap U|=|e\cap V|$ for all $e\in E(H)$ and $W=U\cup V$ forms an equal partition of hyperedges. Thus, the result follows from the \Cref{par-part}.
 \end{proof}
 Any unit with a cardinality of at least $2$ induces an equal partition of hyperedges. Let $W_E=\{v_1,v_2,\ldots,v_k\}$ be a unit with $k\ge 2$. Let $k_e=\max\{i\in \mathbb{N}:i\le k\,\text{~and~}i\text{~is even~}\}$. Now $U=\{v_1,\ldots v_{\frac{k_e}{2}}\}$, and $V=\{v_{(\frac{k_e}{2}+1)},\ldots,v_{k_e}\}$ form an equal partition of hyperedges. Similarly, if $W_{E_1},\ldots, W_{E_k}$ are units with $|W_{E_i}|\ge 2$, then there are two distinct $u_i,v_i\in W_{E_i}$ for all $i=1,\ldots,k$. In that case, $U=\{u_i:i=1,\ldots,k\}$, and $V=\{v_i:i=1,\ldots,k\}$ form an equal partition of hyperedges. However, an equal partition of hyperedges is not necessarily associated with units. For example, consider the hypergraph $H$ with $V(H)=\{1,2,3,4,5\}$, and $E(H)=\{e_1=\{1,2,3,5\},e_2=\{1,3,4,5\},e_3=\{1,2,4,5\}\}$. The sets $U=\{1,5\}$, and $V=\{2,3,4\}$ forms an equal partition of hyperedgess in $H$. Here, $|e\cap U|=2=|e\cap V|$ for all $e\in E(H)$.
Since unit contraction removes only the singularity of $A_{G_H}$ associated with units, if an equal partition of hyperedges is not associated with units, then besides being reflected as a singularity of $A_{G_H}$, it is also reflected as a singularity of $A_{G_{H/\mathcal{R}_u(H)}}$. For any $U\subseteq V(H)$, we define $\hat U=\{W_E\in \mathfrak{U}(H):E=E_v(H)\text{~for some~}v\in U\}$.
\begin{thm}
      Let $H$ be a hypergraph. Two disjoint subsets $U=\{u_1,u_2,\ldots,u_p\}$, and $V=\{v_1,v_2,\ldots,v_q\}$ of the vertex set $V(H)$ form an equal partition of hyperedges in $H$ with $ \hat U\cap \hat V=\emptyset$, if and only if $A_{G_{H/\mathcal{R}_u(H)}}\tilde x=0$, where $x:V(H/\mathcal{R}_u(H))\to\mathbb{R}$ is such that $x(W_E)=0$ if $W_E\notin \hat U\cup\hat V$, $x(W_E)=|W_E\cap U|$ if $W_E\in \hat U$, and $x(W_E)=-|W_E\cap V|$ if $W_E\in \hat V$.
\end{thm}
\begin{proof}
    Since $U$, and $V$ form an equal partition of hyperedges, by the \Cref{par-part}, $I_H^T(\chi_U-\chi_V)=0$. Therefore, for all 
$e\in E(H)$, the summation $\sum\limits_{v\in V(H)}i_{ve}\chi_U(v)=\sum\limits_{v\in V(H)}i_{ve}\chi_V(v)$. This leads to $\sum\limits_{v\in U}s_v(e)=\sum\limits_{v\in V}s_v(e)$ for all $e\in E(H)$. If for any unit $W_E$, two distinct vertices $u_i,u_j\in W_E\cap U$, then $s_{v_i}=s_{v_i}$, and $s_{v_i}(e)=s_{W_E}{(\tilde e)}=s_{v_i}(e) $ for all $e\in E(H)$. Thus,  $\sum\limits_{v\in U}s_v=\sum\limits_{v\in V}s_v$ leads to $ \sum\limits_{W_E\in\hat U}|W_E\cap U|s_{W_E}=\sum\limits_{W_E\in\hat V}|W_E\cap V|s_{W_E}$. Therefore, $\hat W=\hat U\cup \hat V$ is a linearly dependent set in the unit contraction $H/\mathcal{R}_u(H)$. Therefore, the \Cref{Ix=0_ldv} leads to $I_{H/\mathcal{R}_u(H)}^Tx=0$, and thus, $A_{G_{H/\mathcal{R}_u(H)}}\tilde x=0$.

Conversely, $A_{G_{H/\mathcal{R}_u(H)}}\tilde x=0$ leads to $I^T_{{H/\mathcal{R}_u(H)}}x=0$. Therefore, for all $\tilde e\in E(H/\mathcal{R}_u(H))$, 
\begin{align*}
    \sum\limits_{W_E\in\mathfrak{U}(H)}x(v)s_{W_E}(\tilde e)=(I^T_{H/\mathcal{R}_u(H)}x)(\tilde e)=0.
\end{align*}
    Thus, $\sum\limits_{W_E\in \hat U}|W_E\cap U|s_{W_E}-\sum\limits_{W_E\in \hat V}|W_E\cap V|s_{W_E}= \sum\limits_{W_E\in \hat U}x(W_E)s_{W_E}+\sum\limits_{W_E\in \hat V}x(W_E)s_{W_E}=  \sum\limits_{W_E\in\mathfrak{U}(H)}x(v)s_{W_E}=0$.
    Therefore, for all $e\in E(H)$, we have $\sum\limits_{v\in U}s_u(e)=\sum\limits_{W_E\in \hat U}|W_E\cap U|s_{W_E}=\sum\limits_{W_E\in \hat V}|W_E\cap V|s_{W_E}=\sum\limits_{v\in V}s_u(e)$. That is, $ \sum\limits_{v\in U}s_u(e)-\sum\limits_{v\in V}s_u(e)=0$, and thus $A_{G_H}\tilde y=0$, where $y=(\chi_U-\chi_V)$. Therefore, by the \Cref{par-part}, $U$ and $V$ form an equal partition of hyperedges in $H$. 
\end{proof}
 \subsubsection{Equal partition of stars}               
          We have discussed the equal partition of hyperedges in a hypergraph $H$, which leads to linearly dependent vertices. Since each hyperedge $e\in E(H)$ becomes a vertex in the dual hypergraph $H^*$, there are similar concepts of these above-mentioned partitions, which lead to linearly dependent collections of hyperedges in $H$. 

Let $H$ be a hypergraph. If $E=\{e_1,\ldots, e_p\}$, and $F=\{f_1,\ldots,f_q\}$ are two disjoint sets of hyperedges such that $|E_v(H)\cap E|=|E_v(H)\cap F|$ for all $v\in V(H)$, then we say $E=\{e_1,\ldots, e_p\}$, and $F=\{f_1,\ldots,f_q\}$ is an \emph{equal partition of stars} in the hypergraph $H$.
\begin{thm}
    Let $H$ be a hypergraph. The collections $E=\{e_1,\ldots, e_p\}$, and $F=\{f_1,\ldots,f_q\}$ of hyperedges form an equal partition of stars if and only if $A_{G_H}\tilde{y}=0$, where $y:E(H)\to\mathbb{R}$ is defined by $y=\chi_E-\chi_F$.
\end{thm}
\begin{proof}
    Suppose that
     $E=\{e_1,\ldots, e_p\}$, and $F=\{f_1,\ldots,f_q\}$ form an equal partition of stars. For any $v\in V(H)$, since $\sum\limits_{e\in E}\chi_e(v)=|E_v(H)\cap E|=|E_v(H)\cap F|=\sum\limits_{e\in F}\chi_e(v)$, the set $E\cup F$ is a collection of linearly dependent hyperedges with a coefficient vector $y=\chi_E-\chi_F$. Therefore, by the \Cref{I_H_lde}, $I_Hy=0$, and $A_{G_H}\tilde{y}=0$.

     Conversely, if $A_{G_H}\tilde{y}=0$, then $I_Hy=0$. Therefore, for all $v\in V(H)$, we have
     $|E_v(H)\cap E|-|E_v(H)\cap F|= \sum\limits_{e\in E}\chi_e(v)-\sum\limits_{e\in F}\chi_e(v)=\sum\limits_{e\in E(H)}i_{ve}(\chi_E-\chi_F)(e)=(I_Hy)(v)=0$. This completes the proof.
\end{proof}
\subsubsection{Covering projection of a hypergraph and linearly dependent vertices}
Let $H$ and $ \bar H$ be two hypergraphs.  A \emph{hypergraph homomorphism} from $H $ to $\bar H$ is a function $f:V(H)\to V(\bar H)$ such that for all $e\in E(H)$, the collection $\{f(v):v\in e \}\in E(\bar H)$. That is, $f$ induce a map $ \bar{f}:E(H)\to E(\bar H)$ such that $\Bar{f}(e)=\{f(v):v\in e \}$ for all $e\in E(H)$. Note that a hyperedge $e\in E(H)$ and its image $\bar{f}(e)\in E(\bar H)$ can have different cardinalities. For example, the unit contraction in $H$ leads to a hypergraph homomorphism $f:V(H)\to V(H/\mathcal{R}_u(H))$ such that $f(v)=W_E$, where $E=E_v(H)$ for all $v\in V(H)$. For this homomorphism $\Bar{f}(e)=\Tilde{e}$ for all $e\in E(H)$. Now, if a hyperedge $e$
contains an unit $W_E$ with $|W_E|>1$, then $|e|>|\Tilde{e}|$.

A \emph{covering projection} from $H$ to $\Bar H$ is a surjective homomorphism $f:V(H)\to V(\Bar{H})$ such that for all $v\in V(H)$, the restriction map $\bar f|_{E_v(H)}:E_v(H)\to E_{f(v)}(\bar H)$ is a bijection. We say $\Bar{H}$ is a covering projection of $H$ under the projection map $f$. For instance, the hypergraph homomorphism $f: V(H)\to V(H/\mathcal{R}_u(H))$, induced by the unit contraction, is a covering projection. In the \Cref{ex-unit}, consider the vertex $9\in V(H)$ and $W_{E_4}\in V(H/\mathcal{R}_u(H))$ (see the \Cref{fig:unit and contraction}). The restriction map $\bar f|_{E_v(H)}:E_v(H)\to E_{f(v)}(H/\mathcal{R}_u(H))$ is defined by $e\mapsto \Tilde{e}$ for all $e\in \{e_5,e_4\}$. In this instance, the map  $ \bar{f}:E(H)\to E(H/\mathcal{R}_u(H))$ is bijection and thus, all the restriction maps $\bar f|_{E_v(H)}:E_v(H)\to E_{f(v)}(H/\mathcal{R}_u(H))$ are also bijections The following example shows that there are some instances where $ \bar{f}:E(H)\to E(\bar H)$ is not bijection but the restrictions $\bar f|_{E_v(H)}:E_v(H)\to E_{f(v)}(\bar H)$ are bijections for all $v\in V(H)$.
\begin{exm}
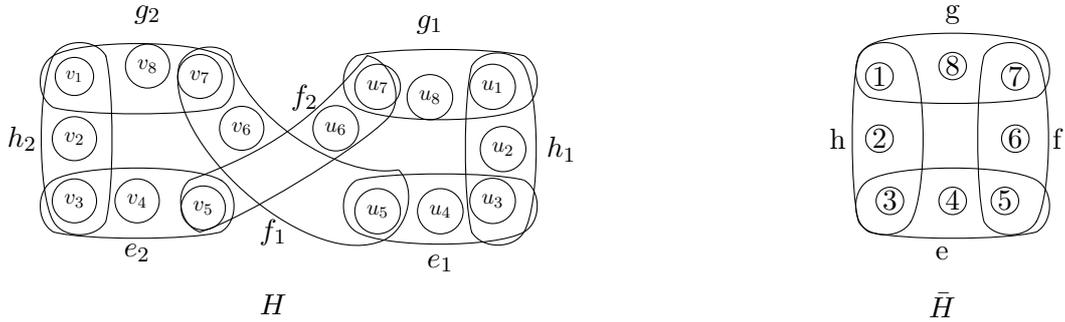
\begin{figure}[H]
    \centering
    \begin{tikzpicture}[scale=0.55]

		\node [style=none] (0) at (-11, 2) {};
		\node [style=none] (1) at (-12.25, 2) {};
		\node [style=none] (2) at (-12.25, -2) {};
		\node [style=none] (3) at (-11, -2) {};
		\node [style=none] (4) at (-8.25, 0.75) {};
		\node [style=none] (5) at (-8.25, 2) {};
		\node [style=none] (6) at (-12.25, 2) {};
		\node [style=none] (7) at (-12.25, 0.75) {};
		\node [style=none] (8) at (-8.25, -2.25) {};
		\node [style=none] (9) at (-8.25, -1) {};
		\node [style=none] (10) at (-12.25, -1) {};
		\node [style=none] (11) at (-12.25, -2.25) {};
		\node [style=none] (12) at (-2.25, -2.25) {};
		\node [style=none] (13) at (-1, -2.25) {};
		\node [style=none] (14) at (-1, 1.75) {};
		\node [style=none] (15) at (-2.25, 1.75) {};
		\node [style=none] (16) at (-5, -1) {};
		\node [style=none] (17) at (-5, -2.25) {};
		\node [style=none] (18) at (-1, -2.25) {};
		\node [style=none] (19) at (-1, -1) {};
		\node [style=none] (20) at (-5, 2) {};
		\node [style=none] (21) at (-5, 0.75) {};
		\node [style=none] (22) at (-1, 0.75) {};
		\node [style=none] (23) at (-1, 2) {};
		\node [style=none] (24) at (-4.25, 0.5) {};
		\node [style=none] (25) at (-4.75, 2) {};
		\node [style=none] (26) at (-9, -1) {};
		\node [style=none] (27) at (-8.75, -2.25) {};
		\node [style=none] (28) at (-4.5, -2.5) {};
		\node [style=none] (29) at (-4, -0.75) {};
		\node [style=none] (30) at (-8, 2) {};
		\node [style=none] (31) at (-9.25, 1.75) {};
		\node [style=none] (32) at (11.25, -2.25) {};
		\node [style=none] (33) at (11.25, -1) {};
		\node [style=none] (34) at (7.25, -1) {};
		\node [style=none] (35) at (7.25, -2.25) {};
		\node [style=none] (36) at (11.25, 1) {};
		\node [style=none] (37) at (11.25, 2.25) {};
		\node [style=none] (38) at (7.25, 2.25) {};
		\node [style=none] (39) at (7.25, 1) {};
		\node [style=none] (40) at (10, -2) {};
		\node [style=none] (41) at (11.25, -2) {};
		\node [style=none] (42) at (11.25, 2) {};
		\node [style=none] (43) at (10, 2) {};
		\node [style=none] (44) at (7, -2) {};
		\node [style=none] (45) at (8.25, -2) {};
		\node [style=none] (46) at (8.25, 2) {};
		\node [style=none] (47) at (7, 2) {};
		\node [style=new style 0,scale=0.65] (48) at (-11.75, 1.5) {$v_1$};
		\node [style=new style 0,scale=0.75] (49) at (-11.75, -1.5) {$v_3$};
		\node [style=new style 0,scale=0.75] (50) at (-11.75, 0) {$v_2$};
		\node [style=new style 0,scale=0.75] (51) at (-10.25, -1.5) {$v_4$};
		\node [style=new style 0,scale=0.75] (52) at (-8.67, -1.67) {$v_5$};
		\node [style=new style 0,scale=0.75] (53) at (-8.75, 1.5) {$v_7$};
		\node [style=new style 0,scale=0.75] (54) at (-10, 1.75) {$v_8$};
		\node [style=new style 0,scale=0.75] (55) at (-7.75, 0.25) {$v_6$};
		\node [style=new style 0,scale=0.75] (56) at (-5.5, 0.25) {$u_6$};
		\node [style=new style 0,scale=0.75] (57) at (-4.5, 1.25) {$u_7$};
		\node [style=new style 0,scale=0.75] (58) at (-3.25, 1) {$u_8$};
		\node [style=new style 0,scale=0.75] (59) at (-1.75, 1.25) {$u_1$};
		\node [style=new style 0,scale=0.75] (60) at (-1.5, -0.25) {$u_2$};
		\node [style=new style 0,scale=0.75] (61) at (-1.75, -1.5) {$u_3$};
		\node [style=new style 0,scale=0.75] (62) at (-3, -1.75) {$u_4$};
		\node [style=new style 0,scale=0.75] (63) at (-4.5, -1.75) {$u_5$};
		\node [style=new style 0] (64) at (7.5, 1.5) {};
		\node [style=new style 0] (65) at (9.25, 1.75) {};
		\node [style=new style 0] (66) at (10.75, 1.5) {};
		\node [style=new style 0] (67) at (10.75, 0) {};
		\node [style=new style 0] (68) at (10.5, -1.5) {};
		\node [style=new style 0] (69) at (9.25, -1.5) {};
		\node [style=new style 0] (70) at (7.75, -1.5) {};
		\node [style=new style 0] (71) at (7.5, 0) {};
		\node [style=none] (72) at (-7, -4) {};
		\node [style=none] (73) at (9, -4) {};
		\node [style=none] (74) at (-7, -4) {$H$};
		\node [style=none] (75) at (9, -4) {$\Bar{H}$};
		\node [style=none] (76) at (7.5, 1.5) {$1$};
		\node [style=none] (77) at (7.5, 0) {$2$};
		\node [style=none] (78) at (7.75, -1.5) {$3$};
		\node [style=none] (79) at (9.25, -1.5) {$4$};
		\node [style=none] (80) at (10.5, -1.5) {$5$};
		\node [style=none] (81) at (10.75, 0) {$6$};
		\node [style=none] (82) at (10.75, 1.5) {$7$};
		\node [style=none] (83) at (9.25, 1.75) {$8$};
		\node [style=none,scale=0.75] (84) at (-8.75, 1.5) {};
		\node [style=none] (85) at (9, -2.75) {e};
		\node [style=none] (86) at (11.75, 0) {f};
		\node [style=none] (87) at (9.25, 3) {g};
		\node [style=none] (88) at (6.5, 0) {h};
		\node [style=none] (89) at (-10.25, -2.75) {$e_2$};
		\node [style=none] (90) at (-6.25, 1) {$f_2$};
		\node [style=none] (91) at (-10, 3) {$g_2$};
		\node [style=none] (92) at (-13, 0) {$h_2$};
		\node [style=none] (93) at (-7, -2.25) {$f_1$};
		\node [style=none] (94) at (-3, -3) {$e_1$};
		\node [style=none] (95) at (-0.5, -0.25) {};
		\node [style=none] (96) at (-0.1, -0.25) {$h_1$};
		\node [style=none] (97) at (-3.25, 2.75) {$g_1$};
	
		\draw (0.center)
			 to [bend right=60] (1.center)
			 to [bend right, looseness=0.50] (2.center)
			 to [bend right=60] (3.center)
			 to [bend right=15, looseness=0.50] cycle;
		\draw (4.center)
			 to [bend right=60] (5.center)
			 to [bend right, looseness=0.50] (6.center)
			 to [bend right=60] (7.center)
			 to [bend right=15, looseness=0.50] cycle;
		\draw (9.center)
			 to [bend right, looseness=0.50] (10.center)
			 to [bend right=60] (11.center)
			 to [bend right=15, looseness=0.50] (8.center)
			 to [bend right=60] cycle;
		\draw (12.center)
			 to [bend right=60] (13.center)
			 to [bend right, looseness=0.50] (14.center)
			 to [bend right=60] (15.center)
			 to [bend right=15, looseness=0.50] cycle;
		\draw [bend right=60] (16.center) to (17.center);
		\draw [bend right, looseness=0.50] (17.center) to (18.center);
		\draw [bend right=60] (18.center) to (19.center);
		\draw [bend right=15, looseness=0.50] (19.center) to (16.center);
		\draw (21.center)
			 to [bend right, looseness=0.50] (22.center)
			 to [bend right=60] (23.center)
			 to [bend right=15, looseness=0.50] (20.center)
			 to [bend right=60] cycle;
		\draw (25.center)
			 to [bend left=15] (26.center)
			 to [bend right=60] (27.center)
			 to [bend right=15, looseness=0.50] (24.center)
			 to [bend right=60] cycle;
		\draw [bend left=45, looseness=0.75] (29.center) to (30.center);
		\draw [bend right=60] (30.center) to (31.center);
		\draw [bend right=60, looseness=0.75] (31.center) to (28.center);
		\draw [bend right=60] (28.center) to (29.center);
		\draw (33.center)
			 to [bend right, looseness=0.50] (34.center)
			 to [bend right=60] (35.center)
			 to [bend right=15, looseness=0.50] (32.center)
			 to [bend right=60] cycle;
		\draw (37.center)
			 to [bend right, looseness=0.50] (38.center)
			 to [bend right=60] (39.center)
			 to [bend right=15, looseness=0.50] (36.center)
			 to [bend right=60] cycle;
		\draw (41.center)
			 to [bend right, looseness=0.50] (42.center)
			 to [bend right=60] (43.center)
			 to [bend right=15, looseness=0.50] (40.center)
			 to [bend right=60] cycle;
		\draw (45.center)
			 to [bend right, looseness=0.50] (46.center)
			 to [bend right=60] (47.center)
			 to [bend right=15, looseness=0.50] (44.center)
			 to [bend right=60] cycle;
	
\end{tikzpicture}
    \caption{Covering projection from $H$ to $\Bar{H}$ with $|E(H)|\ne |E(\Bar{H})|$}
    \label{fig:8-cycle-proj-4-cycle}
\end{figure}
 Let $H$, and $H'$ be a hypergraph (see the \Cref{fig:8-cycle-proj-4-cycle}) with
    $$V(H)=\{u_1,v_1,u_2,v_2,u_3,v_3,u_4,v_4,u_5,v_5,u_6,v_6,u_7,v_7,u_8,v_8\}\text{~and~}V(\Bar{H})=\{1,2,3,4,5,6,7,8\},$$
    $$E(H) =\{e_1,e_2,f_1,f_2,g_1,g_2,h_1,h_2\},\text{~and~} E(\Bar{H})=\{e,f,g,h\}.$$
    Here $e_1=\{u_3,u_4,u_5\}$, $e_2=\{v_3,v_4,v_5\}$, $f_1=\{u_5,v_6,v_7\}$, $f_2=\{v_5,u_6,u_7\}$, $g_1=\{u_1,u_7,u_8\}$, $g_2=\{v_1,v_7,v_8\}$, $h_1=\{u_1,u_2,u_3\}$, $h_2=\{v_1,v_2,v_3\}$, and $ e=\{3,4,5\}$, $f=\{5,6,7\}$, $g=\{1,7,8\}$, $h=\{1,2,3\}$. The map $f:V(H)\to V(\Bar{H})$, defined by $ u_i\mapsto i$, and $v_i\mapsto i$ for all $i=1,2,,4,5,6,7,8$, is a covering projection from $H$ to $\Bar{H}$.
\end{exm}
A covering projection $f:V(H)\to V(\Bar{H})$ is called \emph{cardinality preserving} if $ |e|=|\Bar{f}(e)|$ for all $e\in E(H)$.
\begin{thm}
    Let $H$ and $\Bar{H}$ be two hypergraphs, and $f: V(H)\to V(\Bar{H})$ be a cardinality preserving covering projection. If $U\subseteq V(\bar{H})$ is a linearly dependent set of vertices in $\Bar{H}$, then $f^{-1}(U)$ is a linearly dependent set of vertices in $H$.
\end{thm}
\begin{proof}
   Since $U\subseteq V(\bar{H})$ is a linearly dependent set of vertices in $\Bar{H}$, by the \Cref{Ix=0_ldv} $I_{\bar{H}}^T\bar x=0$ for a non-zero vector $\bar x:V(\bar{H})\to\mathbb{R}$ such that $supp(\bar x)\subseteq U$. Consider the vector $x:V(H)\to\mathbb{R}$ defined as $ x(u)=\bar x(f(u)) $ for all $u\in V(H)$. Since $I_{\bar{H}}^T\bar x=0$, and $ |e|=|\Bar{f}(e)|$ for all $e\in E(H)$, we have $(I_H^Tx)(e)= \sum\limits_{u\in e}x(u)=\sum\limits_{f(u)\in \bar{f}(e)}\bar{x}(\Bar{f}(u))=(I_{\Bar{H}}^T\bar x)(\bar{f}(e))=0$ for all $e\in E(H)$. Therefore, $I_H^Tx$ with $supp(x)\subseteq f^{-1}(U)$. Consequently, by the \Cref{Ix=0_ldv} $f^{-1}(U)$ is a linearly subset of vertices in $H$.
\end{proof}

\section{Applications}
\subsection{Spectra of hypergraph matrices}\label{app-matrices}
Given a hypergraph $H$, the linearly dependent vertices and hyperedges correspond to linearly dependent rows and columns in the incidence matrix $I_H$, respectively. Since other hypergraph matrices are related to the incidence matrix $I_H$, the spectra of these matrices contain the traces of these linear dependencies.
In the case of a graph $G$, the product of the incidence matrix and its transpose $I_GI_G^T$ is called the signless Laplacian matrix associated with $G$ \cite[Section-2, Equation-1]{signless_cvetcovic}. For a hypergraph $H$, let $w_V:V(H)\to(0,\infty)$, and $w_E:E(H)\to(0,\infty)$ be weight functions on vertices and hypergraphs. The signless Laplacian matrix $Q_H=[q_{uv}]_{u,v\in V(H)}$ of a hypergraph $H$ is a matrix, with its rows and columns indexed by $V(H)$, and $q_{uv}=\sum\limits_{e\in E(H)}w_V(u)i_{ue}i_{ve}w_E(e)$. It follows from the definition that $Q_H=D_VI_HD_EI_H^T$, where $D_V$ is a diagonal matrix whose rows and columns are indexed by $V(H)$ and the $(v,v)$-th entry is $w_V(v)$ for all $v\in V(H)$. Similarly, $D_E$ is a diagonal matrix with its rows and columns indexed by $E(H)$ with its $(e,e)$-th entry is $w_E(e)$ for all $e\in E(H)$. 

\begin{exm}\rm
    Different choices of $w_V$ and $w_E$ lead to different variations of the incidence matrix $Q_H$.

    1. If $w_V(v)=1$ for all $v\in V(H)$, and $w_E(e)=1$ for all $e\in E(H)$, then $Q_H=I_HT_H^T$, the signless Laplacian matrix considered by Cardoso and Trevisan in \cite{Trevisan_signless}.
    
    2. In \cite{my1st}, two positive valued function $\delta_{V(H)}:V(H)\to(0,\infty)$, and $\delta_{E(H)}:E(H)\to(0,\infty)$ are used to define the signless Laplacian operator $\mathcal{Q}$. If we set $w_V(v)=\frac{1}{\delta_{V(H)}(v)}$, and $W_E(e)=\frac{\delta_{E(H)}(e)}{|e|^2}$ for all $v\in V(H), e\in E(H)$, then $Q_H$ becomes the matrix representation of $\mathcal{Q}$ with respect to the basis $\{\mathbf{1}_v:v\in V(H)\}$, where $\mathbf{1}_v:V(H)\to\{0,1\} $ is such that $\mathbf{1}_v(u)=1$ if $u=v$, and otherwise $\mathbf{1}_v(u)=0$.
\end{exm}
In the next theorem, we show that a linearly dependent set of vertices corresponds to a $0$ eigenvalue of $Q_H$.
\begin{prop}\label{q-ld}
    Let $H$ be a hypergraph. If $U$ is a linearly dependent set of vertices in $H$, with a coefficient vector $c_U$, then $Q_Hc_U=0$. 
\end{prop}
\begin{proof}Since $U$ is a linearly dependent set of vertices in $H$, with a coefficient vector $c_U$, by \Cref{Ix=0_ldv}, $I_H^Tc_U=0$. Therefore, $Q_Hc_U=Q_H=D_VI_HD_EI_H^Tc_U=0$.
\end{proof}
The converse of the above result is true for the signless laplacian described in \cite{Trevisan_signless}, that is, when $Q_H=I_HI_H^T$. In that case, if $Q_Hx=0$ for some non-zero vector, then $x^TI_HI_H^Tx=0$, and thus $I_H^Tx=0$. Therefore, by the \Cref{Ix=0_ldv}, the $supp(x)$ is a linearly dependent set of vertices. Since each unit $W_E$ with $|W_E|>1$ is a linearly dependent set of vertices by the \Cref{q-ld}, for any two distinct vertices $u,v\in W_E$ the equality $Q_Hx_{uv}=0$, and $ S_{W_E}$ lies in the null space of $Q_H$. This fact also follows from a Theorem in \cite{unit}. If the disjoint subset of vertices $U$ and $V$ forms an equal partitions of hyperedges in the hypergraph $H$, then $W=U\cup V$ is a linearly dependent set of vertices with coefficient vector $\chi_U-\chi_V$. Therefore, by the \Cref{q-ld}, $Q_H(\chi_U-\chi_V)=0$.

Given any hypergraph $H$, suppose that $A_H=[a_{uv}]_{u,v\in V(H)}$ is a matrix whose rows and columns are indexed by the vertex set $V(H)$. For all $u\in V(H)$, the diagonal entry $a_{uu}=0$, and for distinct $u,v\in V(H)$, the $(u,v)$-th entry $a_{uv}=q_{uv}$, the $(u,v)$-th entry of $Q_H$.
We refer to the matrix $A_H$ as the adjacency matrix of $H$. 
\begin{exm}\rm Followings are the different variations of hypergraph adjacency due to specific choices of $w_V$ and $w_E$.

1.If $w_V(v)=1$ for all $v\in V(H)$, and $w_E(e)=1$ for all $e\in E(H)$, the hypergraph adjacency becomes the adjacency matrix described in \cite{feng1996spectra}.

2.If $w_V(v)=1$ for all $v\in V(H)$, and $w_E(e)=\frac{1}{|e|-1}$ for all $e\in E(H)$, the hypergraph adjacency becomes the adjacency matrix described in \cite{hg-mat}. If instead of the previous vertex weight, we set $w_V(v)=\frac{1}{|E_v(H)|}$ for all $v\in V(H)$, then for the same choice of $w_E$, the hypergraph adjacency becomes the normalized adjacency matrix described in \cite{hg-mat}.
\end{exm}
In the next theorem, we show that the effects of specific linearly dependent sets of vertices are also manifested in the Adjacency spectra of a hypergraph $H$.
\begin{cor}\label{a-ld}
    Let $H$ be a hypergraph. For any linearly dependent vertices $U\subseteq V(H)$, if $w_V(u)\sum\limits_{e\in E_u(H)}w_E(e)=c_U$, a constant for all $u\in U$, then $-c_U$ is an eigenvalue of $A_H$.
\end{cor}
\begin{proof}
Consider the diagonal matrix $D_H$, whose rows and columns are indexed by $V(H)$, and the $v$-th diagonal entry of the matrix is $w_V(v)\sum\limits_{e\in E_v(H)}w_E(e)$. The adjacency matrix is $A_H=Q_H-D_H$.
    Since $U$ is a linearly dependent set of vertices, there exists a non-zero coefficient vector $x:V(H)\to\mathbb{R}$ such that $\sum\limits_{u\in U}x(u)s_u=0$ and $x(v)=0$ for all $v\notin U$. Therefore, by the \Cref{q-ld}, $Q_Hx=0$. The condition $w_V(u)\sum\limits_{e\in E_u(H)}w_E(e)=c_U$, leads us to $D_Hx=c_Ux $. Therefore $A_Hx=-c_Ux$. This completes the proof.
\end{proof}
Suppose that $W_E$ is a unit in a hypergraph $H$ with $|W_E|>1$. suppose that  $w_V(v)=w$, a constant for all $v\in W_E$.
Since a unit $W_E$ with $|W_E|>1$ is a set of linearly dependent vertices, and $w_V(v)=w$ for all $v\in W_E$ leads to $w_V(u)\sum\limits_{e\in E_v(H)}w_E(e)=w\sum\limits_{e\in E}w_E(e)$, a constant for all $v\in W_E$, the by \Cref{a-ld}, if $W_E\in\mathfrak{U}(H)$ with $|W_E|>1$ and $w_V(v)=w$ for all $v\in W_E$, then $-w \sum\limits_{e\in E}w_E(e)$ is an eigenvalue of $A_H$. This fact is proved independently in \cite{unit}. If two disjoint collections of vertices $U$ and $V$ form an equal partition of hyperedges in the hypergraph $H$, then since $W=U\cup V $ is a linearly dependent set of vertices in $V(H)$. Thus, by the \Cref{a-ld}, if $w_V(u)\sum\limits_{e\in E_u(H)}w_E(e)=c$, then $-c$ is an eigenvalue of $A_H$.
\begin{exm}\rm   Let $H$ be a hypergraph and $W_E\in\mathfrak{U}(H)$ with $|W_E|>1$.
  
    1. If $w_V(v)=1$ for all $v\in V(H)$, and $w_E(e)=1$ for all $e\in E(H),  $ then $-|E|$ is an eigenvalue of $A_H$.
    
    2. If $w_V(v)=1$ for all $v\in V(H)$, and $w_E(e)=\frac{1}{|e|-1}$ for all $e\in E(H),  $ then $-\sum\limits_{e\in E}\frac{1}{|e|-1}$ is an eigenvalue of $A_H$.

    3. If $w_V(v)=\frac{1}{|E_v(H)|}$, for all $v\in V(H)$, and $w_E(e)=\frac{1}{|e|-1}$ for all $e\in E(H),  $ then $-\frac{1}{|E|}\sum\limits_{e\in E}\frac{1}{|e|-1}$ is an eigenvalue of $A_H$.
\end{exm}
In the case of a graph $G$, the graph Laplacian is $K_G-A_G$, where $K_G$ is the degree matrix of $G$, that is, a diagonal matrix whose row and columns are indexed by $V(G)$ and the $v$-th diagonal entry  is the degree of $v$. Similarly for a hypergraph $H$ the Laplacian matrix  $L_H=[l_{uv}]_{u,v\in V(G)}$ is defined by $L_H=K_H-A_H$, where $K_H=[r_{uv}]_{u,v\in V(H)}$ is a diagonal matrix whose row and columns are indexed by $V(H)$ is such that the $v$-th diagonal entry  
$r_{vv}=(A_H\chi_{V(H)})(v)$, where $\chi_{V(H)}:V(H)\to\mathbb{R}$ is such that $\chi_{V(H)}(v)=1$ for all $v\in V(H)$. Since $L_H=K_H+D_H-Q_H$, we have the following Corollary of the \Cref{q-ld}.
\begin{cor}
    Let $H$ be a hypergraph. If $U$ is a linearly dependent set of vertices in $H$, and $$w_V(v)\left(\sum\limits_{u\in V(H)}\sum\limits_{e\in E_u(H)\cap E_v(H)}w_E(e)\right)=c
_U,\text{~a constant, for all }u\in U,$$
then $c_u$ is an eigenvalue of $L_H$.
\end{cor}
\begin{proof}
     Since $U$ is a linearly dependent set of vertices, there exists a non-zero coefficient vector $x:V(H)\to\mathbb{R}$ such that $\sum\limits_{u\in U}x(u)s_u=0$ and $x(v)=0$ for all $v\notin U$. Therefore, by the \Cref{q-ld}, $Q_Hx=0$. The $v$ the diagonal entry of the diagonal matrix $K_H+D_H$ is $w_V(v)\left(\sum\limits_{u\in V(H)}\sum\limits_{e\in E_u(H)\cap E_v(H)}w_E(e)\right)$. Therefore the condition $w_V(v)\left(\sum\limits_{u\in V(H)}\sum\limits_{e\in E_u(H)\cap E_v(H)}w_E(e)\right)=c
_U$ leads to $L_Hx=c_Ux$. This completes the proof.
\end{proof}
\subsection{Random walk on hypergraphs}\label{rwh}
Let $H$ be a hypergraph with $V(H)$.
By \cite[Theorem 8.1]{Billingsley}, a stochastic matrix $P=\left[P_{uv}\right]_{u,v\in V(H)}$ indexed by $V(H)$, and the collection $\{q_v\in[0,1]:v\in V(H)\}$ with $\sum\limits_{v\in V(H)}q_v=1$ leads us to a  probability space $(\Omega,\mathfrak{F},\Pr)$, and a  \emph{random walk} $\mathcal{X}=\{X_t:t\in \mathbb{T}\}$ on that probability space. Here $\Omega=V(H)^{\mathbb{T}}$, the collection of all the functions $\omega:\mathbb{T}\to V(H)$ with domain $\mathbb{T} =\mathbb{N}\cup \{0\}$, and co-domain $V(H)$. The collection of events $\mathfrak{F}$ is the smallest $\sigma$-algebra containing the collection $C=\bigcup_{t\in\mathbb{T}}C_t$, where $C_t=\{\{x:x(0)=v_{i_0},\ldots,x(t)=v_{i_t}\}\text{~for all~}v_{i_0},\ldots,v_{i_t}\in V(H)\}$, and the probability, $\Pr:\mathfrak{F}\to[0,1]$ is such that $\Pr(\{x:x(0)=v_0,\ldots,x(t)=v_t\})=q_{v_0}P_{v_0v_1}\ldots P_{v_{t-1}v_t}$. For all $t\in\mathbb{T}$, the function $X_t:\Omega\to V(H)$ such that $X_t(w)=v$ if $w(t)=v$ for all $v\in V(H)$. We denote the event $X_t^{-1}(v)=\{\omega\in\Omega:\omega(t)=v\}$ as $(X_t=v) $. The stochastic matrix $P$ is called the \emph{probability transition matrix} of the random walk $\mathcal{X}$, and $\Pr(X_t=v|X_{t-1}=u)=P_{uv}$ for all $t\in \mathbb{N}$, for all $u,v\in V(H)$. The initial distribution is given by $\Pr(X_0=v)=q_v$ for all $v\in V(H)$.
 We refer the readers to \cite{Billingsley,grimmett2020probability} for more details on the preliminaries of random walk. A transition from a vertex $u$ to $v$ corresponds to two intermediate transitions. The first one is $u$ to $e$, where $e\in E_u(H)$, and the second one is $e$ to $v$, for some $v\in e$. Thus, $P_{uv}=\sum\limits_{e\in E_u(H)\cap E_v(H)}r_{ue}s^u_{ev}$ with $\sum\limits_{e\in E_u(H)}r_{ue}=1$, and $\sum\limits_{v\in e}s^u_{ev}=1$ for all $u\in V(H) $, and $e\in E(H)$.  
 We refer to the distribution $\{r_{ue}:e\in E_u(H)\}$ as the \emph{hyperedge distribution} on the vertex $u$, and the distribution $\{s^u_{ev}:v\in e\}$ as the \emph{vertex distribution of a hyperedge $e(\in E_u(H))$ on $u$}. 
 The probability distribution is called uniform if  $r_{ue}=r_{ue'} $ for all $e,e'\in E_u(H)$, and $s^u_{ev}=s^u_{ev'} $ for all $v,v'\in e\setminus\{u\}$. A random walk is said to be \emph{non-lazy }if $s^u_{eu}=0$ for all $u\in V(H)$. Otherwise, it is called a \emph{lazy} random walk
 
 \begin{exm}\label{ex-uniform-rand}\rm
     1. In \cite{hg-mat}, a normalized adjacency matrix $\mathcal{A}=\left(\mathfrak{a}_{uv}\right)_{u,v\in V(H)}$ is used as the probability transition matrix, where $\mathfrak{a}_{uv}=\frac{1}{|E_u(H)|}\sum\limits_{e\in E_u(H)\cap E_v(H)}\frac{1}{|e|-1}$ for two distinct $u,v\in V(H)$, diagonal entry $\mathfrak{a}_{uu} =0$ for all $u\in V(H)$. Since the diagonal entries are $0$, the random walk is \emph{non-lazy}, that is $\Pr(X_t=u|X_{t-1}=u)=0$ for all $u\in V(H)$ and for all $t\in \mathbb{T}$. Here $ r_{ue}=\frac{1}{|E_u(H)|}$ for all $e\in E_u(H)$, the vertex distribution probability is $ s_{ev}^u=\frac{1}{|e|-1}$ for all $v(\ne u)\in e$, and $ s_{eu}^u=0$. Thus, the vertex distribution and hyperedge distributions are uniform.

     2. Similarly, in a lazy random walk, with uniform vertex and hyperedge distribution on a hyperedge $H$, we have $ r_{ue}=\frac{1}{|E_u(H)|}$ for all $e\in E_u(H)$, the vertex distribution probability is $ s_{ev}^u=\frac{1}{|e|}$ for all $v\in e$. Therefore, the transition probabilities are $P_{uv}=\frac{1}{|E_u(H)|}\sum\limits_{e\in E_u(H)\cap E_v(H)}\frac{1}{|e|-1}$ for all $u,v\in V(H)$.
 \end{exm}
 
 Now, we can show that specific linearly dependent vertices induce intriguing properties of the probability transition matrices. For some $U\subseteq V(H)$, we denote the event $\bigcup\limits_{u\in U}(X_t=u)$ as $(X_t\in U)$, and similarly, $(X_t\notin U)=\bigcup\limits_{u\in V(H)\setminus U}(X_t=u)$.
\begin{thm}
    \label{par-part-rw}
      Let $H$ be a hypergraph, and  $U$ and $V$ form an equal partition of hyperedges in $H$. If $\{X_t:t\in\mathbb{T}\}$ is a random walk on $H$ with uniform probability distribution, then $\Pr(X_t\in U|X_{t-1}\notin U\cup V )=\Pr(X_t\in V|X_{t-1}\notin U\cup V )$ for all $t\in \mathbb{N}$.
\end{thm}
 \begin{proof}
 First, we show that for all $w\in V(H)\setminus (U\cup V)$ we have $\sum\limits
_{u\in U}P_{wu}=\sum\limits
_{v\in V}P_{wv}$.
 If $w\in V(H)\setminus(U\cup V)$,
     $$\sum\limits_{u\in U}P_{wu}=\sum\limits_{u\in U}\sum\limits_{e\in E_w(H)\cap E_u(H)}r_{we}s_{eu}^w=\sum\limits_{e\in E_w(H)}\sum\limits_{u\in e\cap U}r_{we}s_{ev}^w.$$
      Since $U$ and $V$ form an equal partition of hyperedges in $H$, $|e\cap U|=|e\cap V|$ for all $e\in E(H)$. Suppose that $e\cap U=\{u_1,\ldots,u_{k_e}\}$, and  $e\cap V=\{v_1,\ldots,v_{k_e}\}$ for all $e\in E_w(H)$. Since $w\notin U\cup V$, and  the probability distribution is uniform, $s_{eu_i}^w=s_{ev_i}^w$ for all $i=1,\ldots,k_e$.
     Consequently,
     $$\sum\limits_{u\in U}P_{wu}=\sum\limits_{e\in E_w(H)}\sum\limits_{i=1}^{k_e}r_{we}s_{eu_i}^w=\sum\limits_{e\in E_w(H)}\sum\limits_{i=1}^{k_e}r_{we}s_{ev_i}^w  =\sum\limits_{v\in V}P_{wv}.$$
     Therefore, $\Pr(X_t\in U|X_{t-1}\notin U\cup V )=\sum\limits_{w\in V(H)\setminus (U\cup V)} \sum\limits_{u\in U}P_{wu}=\sum\limits_{w\in V(H)\setminus (U\cup V)} \sum\limits_{v\in V}P_{wu}=\Pr(X_t\in V|X_{t-1}\notin U\cup V )$.
 \end{proof}
Therefore,  if $U\cup V$ is an equal partition of hyperedges in a hypergraph $H$, then in any random walk with a uniform probability distribution, the probability of transition of the random walk from outside of  $W=U\cup V$ is equally distributed between $U$, and $V$.
 The \Cref{par-part} leads us to the following corollary of the \Cref{par-part-rw}.
 \begin{cor}
     Let $H$ be a hypergraph, and $\{X_t:t\in\mathbb{T}\}$ be a random walk on $H$ with uniform probability distribution. If $U$, and $V$ are two disjoint subsets of $V(H)$ with $A_{G_H}\tilde x=0$ where $x=\chi_U-\chi_V$, then
     $$\Pr(X_t\in U|X_{t-1}\notin U\cup V)=\Pr(X_t\in V|X_{t-1}\notin U\cup V)\text{~for all~}t\in \mathbb{N}.$$
 \end{cor}
 \begin{proof}
     Since by \Cref{par-part}, the equation $A_{G_H}\tilde x=0$ with $x=\chi_U-\chi_V$ implies $ U$, and $V$ form an equal partition of hyperedges in $H$. Consequently, the \Cref{par-part-rw} leads us to the result.
 \end{proof}

For any $u\in V(H)$, the subset $w\in \bigcap\limits_{j<t}(X_t=u|X_j\ne u)$ of $\Omega$ represents the \emph{event of first hitting $u$ at time $t\in \mathbb{T}$}. Since a unit $W_E$ with cardinality at least $2$ induces some specific equal partitions of hyperedges, for two distinct $u,v\in  W_E$, we can say something stronger than the \Cref{par-part-rw}.
 Since there exists a bijection $\mathfrak{f}_{uv}: \bigcap\limits_{j<t}(X_t=u|X_j\ne u)\to \bigcap\limits_{j<t}(X_t=v|X_j\ne v)$ defined by $\mathfrak{F}_{uv}(\omega_u)=\omega_v$, for each $\omega_u\in \bigcap\limits_{j<t}(X_t=u|X_j\ne u)$. Here $\omega_v:\mathbb{T}\to V(H)$ is such that $\omega_v(t)=\omega_u(t) $ if $\omega_u(t)\notin\{u,v\}$, otherwise $\omega_u(t)=v$ if $\omega_v(t)=u$, and $\omega_u(t)=u$ if $\omega_v(t)=v$. Since $E_u(H)=E=E_v(H)$, if the probability transition matrix is such that for all $w\notin\{u,v\}$,  $r_{ue}=r_{ve}$, and $s^u_{ew}=s^u_{ew}$ for all $e\in E$, then $P_{uw}=P_{vw}$. Similarly, if $s^u_{ew}=s^u_{ew}$ then $P_{wu}=P_{wv}$ for all $w\notin\{u,v\}$. These two conditions on $P$ are satisfied for random walks with uniform distribution. Thus, in these cases, we have
  $\Pr(\{\omega_u\})=\Pr(\{\omega_v\})$. In \cite{unit},
using this probability preserving bijection $\mathfrak{f}_{uv}$, it is proved that $\Pr(\bigcap\limits_{j<t}(X_t=u|X_j\ne u)))=\Pr(\bigcap\limits_{j<t}(X_t=v|X_j\ne v)$. The {expected hitting time of $v\in V(H)$} after starting at $w$, $\mathbb{E}_w^v$ is defined as $\mathbb{E}_w^v=\sum\limits_{t=1}^\infty t\Pr(\bigcap\limits_{j<t}(X_t=v,X_0=w|X_j\ne v))$. For random walks with uniform distributions that are described in the \Cref{ex-uniform-rand}, we have $\mathbb{E}_w^u=\mathbb{E}_w^v $ when $u,v\in W_E$, and $w\notin\{u,v\}$(\cite{unit}). 

 \subsection{Hypergraph Centrality}\label{cent}
The centrality of a network is a real-valued function on the vertices (or edges) of the network, which encodes the importance of each vertices (or edges) in the network. We refer the reader to \cite{bavelas1950communication,beauchamp1965improved,borgatti2006graph,newman2008mathematics} for various studies in network centrality.
A vertex centrality in a hypergraph $H$ is a function $x:V(H)\to(0,\infty)$ such that $x(v)$ represents the importance of the vertex $v$. Since "important" is a relative word, a centrality function $x$ is defined according to the sense of the importance captured by the centrality. Similarly, a hyperedge centrality is a function $\alpha: E(H)\to (0,\infty)$ such that $\alpha(e)$ represents the importance of $e$ in the hypergraph $H$. Suppose that $c_V:V(H)\to[0,\infty)$ and $c_E:E(H)\to[0,\infty)$ are, respectively, vertex and hyperedge centrality in $H$ such that 
$c_{V}(u)\propto \sum\limits_{e\in E_u(H)}c_{E}(e)w(e)$ and $c_{E}(e)\propto\sum\limits_{u\in e}c_{V}(u)$. Here, $w:E(H)\to (0,\infty)$ is a function and $w(e)$ is the weight of the hyperedge $e$. That is, the importance of a vertex depends on the importance of all the hyperedges incident with the vertex, and the centrality of a hyperedge depends on all the vertices belonging to the hyperedge. Therefore, there exist two positive reals $k,\gamma$ such that $c_{V(H)}(u)=k \sum\limits_{e\in E_u(H)}c_{E(H)}(e)w(e)$, and $ c_{E(H)}(e)=\gamma\sum\limits_{u\in e}c_{V(H)}(u)$. This leads to $ c_{V}(u)=k\gamma \sum\limits_{e\in E_u(H)\cap E_v(H)}w(e)c_{V}(v)$. Therefore, $c_V$ is an eigenvector of the matrix $M_V=\left[\sum\limits_{e\in E_u(H)\cap E_v(H)}w(e)\right]_{u,v\in V(H)}$, which is a variation of the signless Laplacian matrix $Q_H$. Since $M_V$ is a symmetric matrix with positive entries and the positive valued function $ c_V$ is an eigenvector of $M_H$. By Perron-Frobenius Theorem \cite[p-534, Theorem 8.4.4]{mat-anal-h-j}, $ c_V$ is the Perron vector of $M_V$, and the eigenvalue $k\gamma$ is the spectral radius of $M_V$. Given any unit $W_E$ if $|W_E|>1$, then $W_E$ is a linearly dependent set of vertices with any vector in $S_{W_E}$ is a coefficient vector of $W_E$. By the \Cref{q-ld}, $M_Hx=0$ for all $x\in S_{W_E}$. Since $M_H$ is a symmetric matrix, and for all $x\in S_{W_E}$, the centrality function $c_V$, and $x$ are eigenvectors of $ M_V$ corresponding to two different eigenvalues $0$, and $k\gamma$. Therefore, $\sum\limits_{v\in V(H)}x(v)c_V(v)=0$. For two distinct $u,v\in W_E$ we have  $x_{uv}=\chi_{\{u\}}-\chi_{\{v\}}\in S_{W_E}$. Thus, $\sum\limits_{w\in V(H)}x_{uv}(w)c_V(w)=0$, and that leads to $ c_V(u)=c_V(v)$. Therefore, $c_V$ is constant on each unit of $H$.

 Given a random walk on a hypergraph $H$, the \textit{random walk closeness centrality} of a hypergraph $H$ is a function $c_r:V(H)\to [0,\infty)$ defined by
	 $c_r(v)=\frac{|V(H)|}{\sum\limits_{u\in V(H)}\mathbb{E}_u^v}$.
   The random walk closeness centrality, $c_r(v)$ of $v$, measures how close the vertex $v$ is from all the other vertices.  Let $W_E$ be a unit in $H$. If $\Pr(X_0=u)=\Pr(X_0=v)$  for two distinct $u,v\in W_E$, then using the probability preserving bijection, it can be proved that $\mathbb{E}_v^u=\mathbb{E}_u^v $. Since $\mathfrak{f}_{uv}(\bigcap\limits_{j<t}(X_t=v|X_j\ne v,X_0=u))=\bigcap\limits_{j<t}(X_t=u|X_j\ne u,X_0=u)$, and $\mathfrak{f}_{uv}$ is a probability preserving map, we have $\mathbb{E}_v^u=\mathbb{E}_u^v $.
   Now, $\mathbb{E}_w^u=\mathbb{E}_w^v $ for $u,v\in W_E$, and $w\notin\{u,v\}$, and $\mathbb{E}_v^u=\mathbb{E}_u^v $ lead us to $c_r(u)=c_r(v)$. That is, 
for any
    random walk on $H$ with uniform distribution, the random walk closeness centrality is constant on each unit in $H$. 

    Given a random walk on a hypergraph $H$, the random walk based betweenness centrality, $c_B:V(H)\to (0,\infty)$ is defined as $c_B(w)=\sum\limits_{u,v\in V(H)}\frac{\sum\limits_{t\in \mathbb{T}}Prob(X_t=v|X_0=u,X_i=w,\text{~for some~}i< t)}{\sum\limits_{t\in \mathbb{T}} Prob(X_t=v|X_0=u)}$   for any $w\in V(H)$. Suppose that $W_E$ is a unit, and $u,v\in W_E$. If the probability distribution associated with the random walk is uniform, then for all $w\notin \{u,v\}$, $\mathfrak{f}_{uv}((X_t=v|X_0=u,X_i=w,\text{~for some~}i< t))=(X_t=u|X_0=v,X_i=w,\text{~for some~}i< t)$, and $\mathfrak{f}_{uv}((X_t=v|X_0=u))=(X_t=u|X_0=v)$, and the map $\mathfrak{f}_{uv}$ is probability preserving, $c_B(u)=c_B(v)$ for all $u,v\in W_E$. Therefore, for a random walk with a uniform probability distribution on a hypergraph, the random walk-based betweenness centrality is constant on each unit in the hypergraph. 

Let $H$ be a hypergraph. A pair of distinct units $W_{E_i},W_{E_j}\in\mathfrak{U}(H)$ are called \emph{unit-neighbours} if $W_{E_i}\cup W_{E_j}\subseteq e$ for some $e\in E(H)$. Note that this definition suggests that no unit is unit-neighbour of itself. The graph projection of a hypergraph $H$ is a graph $\Gamma_H$ such that $V(\Gamma_H)=\mathfrak{U}(H)$ and $E(\Gamma_H)=\{\{W_{E_i},W_{E_j}\}:W_{E_i},W_{E_j}\text{ are unit-neighbours}\}$. For the notions of a walk, path and distance in a graph, we refer the reader to \cite[Chapter-2]{harary1969graph}. A path of length $n$, in a graph $G$ is an alternating sequence $v_0e_1v_1e_2\ldots v_{n-1}e_nv_n $ where $v_0,v_1,\ldots,v_n\in V(G)$, $e_1,e_2,\ldots,e_n\in E(H)$, and if $i\ne j$  then $v_i$, $v_j$ are distinct. The distance $d_G:V(G)\times V(G)\to [0,\infty)$ is such that $d_G(u,v)=$ the length of shortest path between $u,v$ for two distinct vertices $u,v$, and $d_G(u,u)=0$ for all vertex $u$.
Suppose that $H$ is a hypergraph. We define the \emph{status of} $v\in V(H)$, denoted by $s(v)$, as $s(v)=\sum\limits_{u\in V(H)}\hat d_H(u,v)$. For a hypergraph $H$, the metric $d_{\Gamma_H}$ on $V(\Gamma_H)$ induces a function $\hat{d}_H:V(H)\times V(H)\to [0,\infty)$ defined as
$\hat d_{H}(u,v)=d_{\Gamma_H}(W_{E_u(H)},W_{E_v(H)}).$
Since $\hat d_{H}(u,v)=0$ for two distinct vertex $u,v\in W_{E_0} $, for some $W_{E_0}\in \mathfrak{U}(H)$, $\hat{d}$ is a pseudo metric on $V(H)$.
For any connected graph $G$, the closeness centrality $cl_G:V(G)\to \mathbb{R}$ is defined as $cl_G(v)=(\sum\limits_{u\in V(G)}\ d_G(u,v))^{-1}$ (see \cite[Equation.3.2]{koschutzki2005centrality}).
The higher value of $s(v)$ indicates the larger unit distances of $v$ from the other vertices of $H$. Thus, for a connected hypergraph $H$ with at least two hyperedges, we define the \emph{unit based closeness centrality} $cl_{\mathfrak{U}(H)}:V(H)\to\mathbb{R}$ as  $cl_{\mathfrak{U}(H)}(v)=s(v)^{-1}$ for all $v\in V(H)$.
For any $v\in V(H)$, the higher value of $cl_{\mathfrak{U}(H)}(v)$ indicates smaller unit distances of $v$ from all other vertices of $H$.  The vertex centrality $cl_{\mathfrak{U}(H)}$ is constant on $W_{E}$ for all $W_{E}\in \mathfrak{U}(H)$. The \emph{unit-eccentricity} of $v\in V(H)$ is defined by
$e_H(v)=\max\{\hat d_H(v,u):u\in V(H)\}.$ The unit-eccentricity is constant on each unit of a hypergraph.

    Thus, for specific hypergraph centralities, we have seen units are some clusters of vertices with equal centralities. Therefore, these vertex centralities can be represented as a function of the form $\hat x:\mathfrak{U}(H)\to(0,\infty)$. Given any specific vertex centrality $c: V(H)\to \mathbb{R}$, the $c$-center of $H$ is the collection of vertices $\{v\in V(H):c(v)=\max\{{c(u):u\in V(H)\}}\}$. Now, if $c$ is constant on each unit, then the $c$-center of $H$ is a union of units. Thus, in this section, we have seen some centrality $c$ such that the $c$-center of $H$ is the union of units.
\subsection{Counting methods and linear dependent vertices and hyperedges}\label{count}
For some natural number $n$, and  $1\le k\le n-1$, the Pascal's formula says
$$\binom{n}{k}=\binom{n-1}{k}+\binom{n-1}{k-1}.$$ 
A combinatorial proof of Pascal's formula can be found in \cite{combinatorics-brualdi}. This proof is based on the fact that $\binom{n}{k}$ is the total number of $k$-element subsets of an $n$-element subset $N$, and a particular object $a(\in N)$ belongs to $\binom{n-1}{k-1} $ of the $k$-sets, and $a$ does not belong to $\binom{n-1}{k}$ of them. This proof can be rewritten in terms of linearly dependent hyperedges of a hypergraph. Let $N=\{1,2,\ldots,n\}$.
Consider the hypergraph $H$ with $V(H)=2^N$, the power set of $N$ and $E(H)=\{e_0,e_1,e_2\}$, where $e_0=\binom{n}{k}$, the collection of all the $k$-element subsets of $N$, $e_1=\{A\in \binom{n}{k}:a\in A \}$, and  $e_2=\{A\in \binom{n}{k}:a\notin A \}$ for some fixed $a\in N$. The hyperedge $e_0$ can be written as the disjoint union $e_1\cup e_2$. Thus, $E(H)$ is linearly dependent and $\chi_{e_0}=\chi_{e_1}+\chi_{e_2}$. Therefore, $(\chi_{e_0})(e)=(\chi_{e_1}+\chi_{e_2})(e) $, and $\binom{n}{k}=\sum\limits_{e\in E(H)}\chi_{e_0}(e)=\sum\limits_{e\in E(H)}(\chi_{e_1}+\chi_{e_2})(e)=\binom{n-1}{k}+\binom{n-1}{k-1} $.

The simplest form of the Pigeonhole Principle says that if $n+1$ objects are distributed in $n$ boxes, then at least one box contains two or more objects. This can be proved as a consequence of the \Cref{0-ev-ld}. Consider the hypergraph in which non-empty boxes are the hyperedges, and the objects are vertices. Thus, the number of vertices $|V(H))|=n+1$, and $|E(H)|\le n$. Therefore, the rank of the incidence matrix $I^T_H$ can be at most $n$. Therefore, there exists a non-zero vector $x:V(H)\to\mathbb{R}$ such that $I^T_Hx=0$. Therefore, $A_{G_H}\Tilde{x}=0$. Therefore, by \Cref{0-ev-ld} the vertices are linearly dependent and without loss of generality, we can assume $s_{v_0}=\sum\limits_{i=1}^nc_is_{v_i}$, where $V(H)=\{v_0,v_1,\ldots,v_n\}$. Since each object is distributed in some box, $s_{v_0}$ is non-zero, that is, $s_{v_0}(e)=1$ for some $e\in E(H)$. Consequently, $s_{v_0}=\sum\limits_{i=1}^nc_is_{v_i}$ implies that $s_{v_i}(e)\ne 0$ for some $i\in \{1,\ldots,n\}$. Thus, $v_0,v_i\in e$, and this proves the above-mentioned simple form of the Pigeonhole Principle.

 Given $n$ positive integers $q_1,\ldots,q_n$, suppose that $q_1+\ldots+q_n-n+1$ objects are distributed into $n$ boxes. Consider the hypergraph $H$, in which the objects are vertices, and each non-empty box is a hyperedge. Since each object can be contained in at most one box, $E_v(H)$ is a singleton set for each $v\in V(H)$. Let $E(H)=\{e_1,\ldots,e_m\}$ for some $m\le n$. Let the unit $W_{E_i}=\{v\in V(H):E_v(H)=\{e_i\}\}$ for all $i=1,\ldots,m$. Consequently, the unit-contraction $H/\mathcal{R}_u(H)$ has $m$ vertices and $m$ hyperedges. If $|W_{E_i}|<q_i$ for all $i=1,2,\ldots,m$, then $|V(H)|\le q_1+\ldots+q_m-m+1<q_1+\ldots+q_n-n+1$, a contradiction. Therefore, given $n$ positive integers $q_1,\ldots,q_n$, if $q_1+\ldots+q_n-n+1$ objects are distributed into $n$ boxes, then either the first box contains at least $q_1$, or the second box contains $q_2$ objects, $\ldots,$ or the $n$-th box contains at least $q_n$ objects. 
This is called the stronger version of the Pigeonhole Principle, a generalization of the simple form of the Pigeonhole Principle.

For three non-empty set $A$, $B$, and $C$ the principle of inclusion–exclusion states that 
\begin{equation}\label{ex-in-3}
    |A\cup B\cup C|=(|A|+|B|+|C|)-(|A\cap B|+|B\cap C|+|C\cap A|)+(|A\cap B\cap C|).
\end{equation}
Consider a hypergraph $H$ with $V(H)=A\cup B\cup C$, and $E(H)=\{A\cup B\cup C, A\cap B, B\cap C, C\cap A, A\cap B\cap C, A, B,C\}$. Here, the set of hyperedges forms an equal partition of stars $E,F$ of hyperedges, where $E=\{A\cup B\cup C,A\cap B\cap C,A, B,C\}$, and $F=\{A\cap B, B\cap C, C\cap A\}$. Consequently, $\chi_{_{A\cup B\cup C}}+\chi_{_{A\cap B\cap C}}+\chi_{_A}+\chi_{_B}+\chi_{_C}=\chi_{_{A\cap B}}+\chi_{_{B\cap C}}+\chi_{_{C\cap A}}$. Therefore,
$\sum\limits_{a\in A\cup B\cup C}(\chi_{_{A\cup B\cup C}}+\chi_{_{A\cap B\cap C}}+\chi_{_A}+\chi_{_B}+\chi_{_C})(a)=\sum\limits_{a\in A\cup B\cup C}(\chi_{_{A\cap B}}+\chi_{_{B\cap C}}+\chi_{_{C\cap A}})(a)$, and this gives the \Cref{ex-in-3}.

\section{Discussion and conclusion}
\label{sec:discussion}
Given a hypergraph $H$, the vectors in the nullspace of $A_{G_H}$ correspond to a linearly dependent set of vertices or hyperedges. Such linear dependencies lead to a specific sub-structure in the hypergraph $H$. Here, we identify some of these substructures and provide the form of the vector in the nullspace of $A_{G_H}$, that characterizes these substructures. We show that the singularity of $A_{G_H}$, due to units within $H$, can be mitigated through unit contraction. However, singularities may arise from sources beyond units, like other linear dependencies among vertices and hyperedges, for which we currently lack prescribed resolution methods to remove that singularity, warranting further exploration.

Here, we show that PHP can be explained by considering the concept of linearly independent hyperedges and vertices. However, we have not established a clear link between these two domains. Hence, it presents an intriguing direction for further investigation. Given that both the PHP and Inclusion-Exclusion principles can be interpreted using linearly independent hyperedges and vertices, it prompts curiosity regarding the potential connections between these notions in hypergraph theory and other counting methods.

We observe that specific hypergraph centralities belong to a subspace of $\mathbb{R}^{V(H)}$ of dimension $|\mathfrak{U}(H)|$, and the existence of any unit with cardinality at least $2$ ensure that this subspace is a proper subspace of $\mathbb{R}^{V(H)}$. This fact indicates that specific hypergraph structures related to linearly independent rows and columns leave their traces in some proper subspaces of $\mathbb{R}^{V(H)}$. These traces would be an intriguing direction to study. For example, if a vertex function $x:V(H)\to\mathbb{R}$ is such that the value $x(v)$ depends on $E_v(H)$, then $x$ can be embedded into an $|\mathfrak{U}(H)|$ dimensional vector space.
\label{sec:conclusion}
	\section*{Acknowledgement} The author acknowledges the financial assistance provided by the National Institute of Science Education and Research through the Department of Atomic Energy plan project RIA 4001 (NISER).


\begin{thebibliography}{10}

\bibitem{hg-mat}
A.~Banerjee.
\newblock On the spectrum of hypergraphs.
\newblock {\em Linear Algebra Appl.}, 614:82--110, 2021.

\bibitem{unit}
A.~{Banerjee} and S.~{Parui}.
\newblock {On some building blocks of hypergraphs}.
\newblock {\em arXiv e-prints}, page arXiv:2206.09764, June 2022.

\bibitem{my1st}
A.~Banerjee and S.~Parui.
\newblock On some general operators of hypergraphs.
\newblock {\em Linear Algebra Appl.}, 667:97--132, 2023.

\bibitem{bavelas1950communication}
A.~Bavelas.
\newblock Communication patterns in task-oriented groups.
\newblock {\em The journal of the acoustical society of America}, 22(6):725--730, 1950.

\bibitem{beauchamp1965improved}
M.~A. Beauchamp.
\newblock An improved index of centrality.
\newblock {\em Behavioral science}, 10(2):161--163, 1965.

\bibitem{Berge-graph-hypergraph}
C.~Berge.
\newblock {\em Graphs and hypergraphs}, volume Vol. 6 of {\em North-Holland Mathematical Library}.
\newblock North-Holland Publishing Co., Amsterdam-London; American Elsevier Publishing Co., Inc., New York, revised edition, 1976.
\newblock Translated from the French by Edward Minieka.

\bibitem{Berge-hypergraph}
C.~Berge.
\newblock {\em Hypergraphs}, volume~45 of {\em North-Holland Mathematical Library}.
\newblock North-Holland Publishing Co., Amsterdam, 1989.
\newblock Combinatorics of finite sets, Translated from the French.

\bibitem{Billingsley}
P.~Billingsley.
\newblock {\em Probability and measure}.
\newblock Wiley Series in Probability and Mathematical Statistics. John Wiley \& Sons, Inc., New York, third edition, 1995.
\newblock A Wiley-Interscience Publication.

\bibitem{borgatti2006graph}
S.~P. Borgatti and M.~G. Everett.
\newblock A graph-theoretic perspective on centrality.
\newblock {\em Social networks}, 28(4):466--484, 2006.

\bibitem{Bretto-hypergraph}
A.~Bretto.
\newblock {\em Hypergraph theory}.
\newblock Mathematical Engineering. Springer, Cham, 2013.
\newblock An introduction.

\bibitem{combinatorics-brualdi}
R.~A. Brualdi.
\newblock {\em Introductory combinatorics}.
\newblock Pearson Prentice Hall, Upper Saddle River, NJ, fifth edition, 2010.

\bibitem{Trevisan_signless}
K.~Cardoso and V.~Trevisan.
\newblock The signless {L}aplacian matrix of hypergraphs.
\newblock {\em Spec. Matrices}, 10:327--342, 2022.

\bibitem{signless_cvetcovic}
D.~Cvetkovi\'{c}.
\newblock Signless {L}aplacians and line graphs.
\newblock {\em Bull. Cl. Sci. Math. Nat. Sci. Math.}, (30):85--92, 2005.

\bibitem{feng1996spectra}
K.~Feng et~al.
\newblock Spectra of hypergraphs and applications.
\newblock {\em Journal of number theory}, 60(1):1--22, 1996.

\bibitem{grimmett2020probability}
G.~Grimmett and D.~Stirzaker.
\newblock {\em Probability and random processes}.
\newblock Oxford university press, 2020.

\bibitem{harary1969graph}
F.~Harary.
\newblock {\em Graph theory}.
\newblock Addison-Wesley publishing company, 1969.

\bibitem{mat-anal-h-j}
R.~A. Horn and C.~R. Johnson.
\newblock {\em Matrix analysis}.
\newblock Cambridge University Press, Cambridge, second edition, 2013.

\bibitem{koschutzki2005centrality}
D.~Kosch{\"u}tzki, K.~A. Lehmann, L.~Peeters, S.~Richter, D.~Tenfelde-Podehl, and O.~Zlotowski.
\newblock Centrality indices.
\newblock In {\em Network analysis}, pages 16--61. Springer, 2005.

\bibitem{newman2008mathematics}
M.~E. Newman.
\newblock The mathematics of networks.
\newblock {\em The new palgrave encyclopedia of economics}, 2(2008):1--12, 2008.

\bibitem{rodriguez2003laplacian}
J.~A. Rodr\'{\i}guez.
\newblock On the {L}aplacian spectrum and walk-regular hypergraphs.
\newblock {\em Linear Multilinear Algebra}, 51(3):285--297, 2003.

\bibitem{Swarup-panda-2022-hypergraph}
S.~S. Saha, K.~Sharma, and S.~K. Panda.
\newblock On the {L}aplacian spectrum of {$k$}-uniform hypergraphs.
\newblock {\em Linear Algebra Appl.}, 655:1--27, 2022.

\bibitem{Sarkar-banerjee-2020}
A.~Sarkar and A.~Banerjee.
\newblock Joins of hypergraphs and their spectra.
\newblock {\em Linear Algebra Appl.}, 603:101--129, 2020.

\bibitem{Voloshin-Vitaly-graph-hypergraph}
V.~I. Voloshin.
\newblock {\em Introduction to graph and hypergraph theory}.
\newblock Nova Science Publishers, Inc., New York, 2009.

\end{thebibliography}


\end{document}